%% file: Kue18_ADIresmin-arxiv.tex
\documentclass{article}
\usepackage[margin=2.3cm,a4paper]{geometry}
\usepackage{amsmath}
\usepackage{amssymb}
\usepackage{amsthm}
\usepackage[utf8]{inputenc}
\usepackage{tikz,pgfplots}
\tikzstyle{every picture}+=[remember picture]
\usepackage[space]{cite}
\usepackage{tabularx}
\usepackage{multirow}
\usepackage{rotating}
\usepackage{graphicx}  
\definecolor{mpigreen}{HTML}{5C871D}    

\usepackage{hyperref}

\usepackage[vlined,longend,linesnumbered,ruled]{algorithm2e}
\SetKwBlock{Dotw}{do twice}{} 
\usepackage[software,hardware,siam]{mymacros}
\newcommand{\smb}{\left[\begin{smallmatrix}}
\newcommand{\sme}{\end{smallmatrix}\right]}
\DeclareMathOperator*{\grad}{grad}
\def\Range{\mathop{\operator@font Range}\nolimits}
\DeclareMathOperator*{\rank}{rank}
\newcommand{\tensorlab}{\mbox{\textsf{Tensorlab}}}
\newcommand{\granso}{\textsf{\textsc{Granso}}}

\DeclareMathOperator*{\myargmin}{argmin}
\theoremstyle{definition}

\newtheorem{theorem}{Theorem}

\newtheorem{remark}{Remark}

\newtheorem{corollary}[theorem]{Corollary}

\pgfplotscreateplotcyclelist{resN}{%
{color=blue, densely dashed, mark=*, mark options={solid,mark size=1.5pt}, mark repeat=5},
{color=blue, densely dashed, mark=triangle*, mark options={solid,mark size=1.5pt}, mark repeat=5},
{color=mpigreen, solid, mark=triangle*, mark options={solid,mark size=1.5pt}, mark repeat=5},
{color=red, solid, mark=*, mark options={solid,mark size=1.5pt}, mark repeat=5},
{color=orange, dashdotted, mark=diamond, mark options={fill=orange,solid}, mark repeat=5},
{color=red, densely dashed, mark=pentagon, mark options={solid}, mark repeat=5},
{color=mpigreen, densely dashed,solid, mark=square*, mark options={solid,mark size=1.5pt}, mark repeat=5},
{color=mpigreen, densely dotted, mark=star, mark options={fill=orange,solid}, mark repeat=5},
{color=red, densely dashed, mark=triangle, mark options={solid}, mark repeat=5},
}

\hypersetup{%
pdftitle = {Residual minimizing shift parameters for the low-rank ADI iteration}, %
pdfsubject = {Matrix Equations, \today}, %
pdfauthor = {Patrick K\"urschner}, %
pdfkeywords = {Lyapunov equation, alternating
direction implicit, rational Krylov subspaces, shift parameters}
}

\title{Residual minimizing shift parameters for the low-rank ADI iteration} %
\author{Patrick K{\"u}rschner\footnotemark[2]}
\begin{document}
\maketitle
\begin{abstract}
The low-rank alternating directions implicit (LR-ADI) iteration is a frequently employed method for efficiently computing low-rank approximate solutions of
large-scale  
Lyapunov equations. In order to achieve a rapid error reduction, the iteration requires 
shift parameters whose selection and generation is often a difficult task, especially for nonsymmetric coefficients in the Lyapunov equation.
This article represents a follow up of Benner et al. [ETNA,~43 (2014–2015), pp. 142--162] and investigates self-generating shift parameters based on
a minimization principle for the Lyapunov residual norm. Since the involved objective functions are too expensive to evaluate and, hence, intractable,
compressed 
objective functions are introduced which are efficiently constructed from the available data generated by the LR-ADI iteration.
Several numerical experiments indicate that these residual minimizing shifts using approximated objective functions outperform existing precomputed and dynamic
shift parameter selection
techniques, although their generation is more involved.
\end{abstract}


                                                 

\footnotetext[2]{Max Planck Institute for
Dynamics of Complex Technical Systems,  Computational Methods in  Systems and
Control Theory, Magdeburg, Germany, {\tt
  kuerschner@mpi-magdeburg.mpg.de}}

 \section{Introduction}
 In this paper, we study the numerical solution of large-scale, continu\-ous-time, algebraic Lyapunov equations (CALE)
 \begin{align}\label{cale}
  AX+XA^*+BB^*=0
 \end{align}
defined by matrices $A\in\Rnn$, $B\in\R^{n\times s}$, $s\ll n$, and $X\in\Rnn$ is the sought solution.
For large sizes $n$ of the problem, directly computing and storing $X$ is infeasible. For dealing with~\eqref{cale}, it has become common practice to
approximate $X$ by a low-rank factorization $X\approx ZZ^*$ with $Z\in\R^{n\times r}$, $\rank{Z}=r\ll n$. Theoretical evidence for the existence of such
low-rank approximations can be found, e.g., in~\cite{Sab07,AntSZ02,Gra04,BecT17}.
The low-rank solution factor $Z$ can be computed by iterative
methods employing techniques from large-scale numerical linear algebra. 
Projection based methods utilizing extended or rational Krylov subspaces, and the
low-rank alternating directions implicit (LR-ADI) iteration, belong to the most successful and often used representatives of iterative low-rank methods 
for~\eqref{cale}, see, e.g.,~\cite{LiW02,morBenKS13,DruKS11,DruS11,Sim16a,BenS16}. 

Here, we focus on the LR-ADI iteration and a particular important issue thereof.
One of the biggest reservations against the LR-ADI iteration is its dependence on certain parameters, called shifts, which steer the convergence rate of the
iteration. For large problems, especially those defined by nonsymmetric matrices $A$, generating these shift parameters is a difficult task and often only
suboptimal or heuristic shift selection approaches can be employed. 
In~\cite{BenKS14b}, a shift generation approach was proposed, where the shifts are chosen dynamically in the course of the LR-ADI iteration
and are based on minimizing the Lyapunov residual norm. Unfortunately, although potentially leading to very good shifts,
this approach is in its original form only a theoretical concept, because employing it is numerically very expensive and, thus, unusable in practice. 
This article follows up on~\cite{BenKS14b} and investigates several aspects and modifications of the residual minimization based shift selection. The main
goal is a numerically feasible and efficient generation of high quality shift parameters for the LR-ADI iteration, that are based on the residual minimization
principle.
\subsection{Notation}
 $\R$ and $\C$ denote the
real and complex numbers, and $\R_-,~\C_-$ refer to the set of
strictly negative real numbers and the open left half plane. In the
matrix case, $\Rnm,~\Cnm$ denote $n\times m$ real and complex
matrices, respectively. For a complex quantity
$X=\Real{X}+\jmath\Imag{X}$, $\Real{X},~\Imag{X}$ are its real and
imaginary parts, and $\jmath$ is the imaginary unit. The complex
conjugate of $X$ is denoted by
$\overline{X}$ and $|\xi|$ is the absolute value of
$\xi\in\C$. If not stated otherwise, $\|\cdot\|$ is the Euclidean vector- or subordinate matrix norm (spectral norm). 
The matrix $A^*$ is the transpose of
a real or the complex
conjugate transpose of a complex matrix, $A^{-1}$ is the inverse of a nonsingular
matrix $A$, and $A^{-*}=(A^*)^{-1}$. The identity matrix of
dimension $n$ is indicated by $I_n$. 
The spectrum of a matrix $A$ is given by $\Lambda(A)$ and the spectral radius is defined as 
$\rho(A):=\max\lbrace|\lambda|,~\lambda\in\Lambda(A)\rbrace$.

For a multivariate function $f(x_1,\ldots,x_d)~:~\R^d\mapsto\R$, we employ the typical shorthand notation $\psi_{x_i}=\partderiv{f}{x_i}{}$ and $\psi_{x_i
x_j}=\frac{\partial^2 f}{\partial x_i\partial x_j}$ for the first and second order partial derivatives, accumulated in gradient  $\grad f=[\psi_{x_i}]$ and
Hessian $\grad^2=[\psi_{x_ix_j}]$, respectively.
For a vector valued function $F(x_1,\ldots,x_d)=[\psi_1,\ldots,\psi_v]^T$, the Jacobian is given by $[\partderiv{\psi_i}{x_j}{}]$. For
complex functions $g(z_1,\ldots,z_d,\overline{z_1},\ldots,\overline{z_d})$ depending on complex variables and their conjugates,  Wirtinger
calculus~\cite{Rem91} is
used to define complex and complex conjugate derivatives, i.e., derivatives with respect to $z_i$ and $\overline{z_i}$, respectively.
\subsection{Problem assumptions}
Throughout the article we assume that $\Lambda(A)\subset\C_-$ which ensures a unique positive semidefinite solution $X$ of~\eqref{cale}. To permit
low-rank approximations of $X$, we shall assume that $s\ll n$. 
Moreover, we assume that we are able
to efficiently solve linear systems of equations of the form $(A+\alpha I)x=b$, $\alpha\in\C$ by either iterative or
sparse-direct solvers, where we restrict ourselves for the sake of brevity to the latter type of solvers.
\subsection{Overview of this article}
We begin by reviewing the low-rank ADI iteration in Section~\ref{sec:lradi}, including important structural properties of the method and a brief
recapitulation
of the ADI shift parameter problem. The residual norm minimizing shift parameters are discussed in depth in Section~\ref{sec:resminshift}, where our main contribution, a
numerically efficient approach to obtain those shifts, is presented. The main building block is the replacement of the expensive to evaluate and intractable
objective
functions by approximations that are constructed from the already computed data.
Along the way, extensions to generalized Lyapunov equations
 \begin{align}\label{gcale}
  AX^*M^*+MXA^*+BB^*=0
 \end{align}
with nonsingular $M\in\Rnn$ will be discussed. Section~\ref{sec:multi} extents these ideas to the generation of a single shift for the use in more than one
LR-ADI iteration steps which can further reduce the computation times.
A series of numerical experiments is given in Section~\ref{sec:numex}, evaluating the performance of the proposed shift generation machinery in comparison with
existing selection strategies. Comparisons with other low-rank algorithms for~\eqref{cale} are also presented.
Section~\ref{sec:summary} concludes the paper and provides some future research directions.
 \section{Review of the low-rank ADI iteration}\label{sec:lradi}
 The low-rank ADI iteration can be derived from the
nonstationary iteration 
  \begin{align}
    \begin{split}
X_j=&(A-\overline{\alpha_j}I)(A+\alpha_jI)^{-1}X_{j-1}(A+\alpha_jI)^{-H}
(A-\overline{\alpha_j}I)^*\\
    &-2\Real{\alpha_j}(A+\alpha_jI)^{-1}BB^*(A+\alpha_jI)^{-H},~j\geq 1,~X_0\in\Rnn
  \end{split}
 \end{align}
 for the CALE~\eqref{cale}. There, $\alpha_i\in\C_-$, $1\leq i \leq j$, are the previously mentioned shift parameters, discussed further in 
Section~\ref{ssec:shifts}.
By introducing low-rank approximations $X_j=Z_jZ_j^*$ in each step and assuming $Z_0=0$, the above iteration can be
rearranged~\cite{LiW02,Saa09,morBenKS13,Kue16} into the low-rank ADI iteration illustrated in Algorithm~\ref{alg:glradi}. 

 \begin{algorithm}[t]
\SetEndCharOfAlgoLine{}
\SetKwInOut{Input}{Input}\SetKwInOut{Output}{Output}
  \caption{LR-ADI iteration for computing low-rank solution factors.}
  \label{alg:glradi}
    \Input{Matrices $A,~B$ defining~\eqref{cale}, tolerance $0<\tau\ll1$.}
    \Output{$Z_{j}\in\C^{n\times sj}$, such that
    $ZZ^*\approx X$.}
  $W_0=B,\quad Z_0=[~],\quad j=1$, choose $\alpha_1$.\;
  \While{$\|W_{j-1}^*W_{j-1}\|\geq\tau\|B^*B\|$\nllabel{alg:resstop}}{%
    Solve $(A+\alpha_jI)V_j=W_{j-1}$ for $V_j$.\;
      $W_j=W_{j-1}-2\Real{\alpha_j}V_j$.\nllabel{alg:glradiW}\;
$Z_j=[Z_{j-1},~\sqrt{-2\Real{\alpha_j}}V_j]$.\;
Select next shift $\alpha_{j+1}\in\C_-$.\;
$j=j+1$.\;
  }
\end{algorithm}

For the Lyapunov residual matrix regarding the approximate solution $X_j=Z_jZ_j^*$ we have the following result.
\begin{theorem}[{\cite{morBenKS13,PanW13a}}]\label{thm:lradiRes}
Assume $j$ steps of the LR-ADI iteration with the shift parameters $\lbrace
\alpha_1,\ldots,\alpha_j\rbrace\subset\C_-$ have been applied to~\eqref{cale}. Then the Lyapunov residual matrix can be factorized via
\begin{align}\label{lradi_res}
    R_j&=A Z_jZ_j^*+Z_jZ_j^*A^*+BB^*=W_jW_j^*,
  \end{align}
where the \textit{residual factor} $W_j\in\Cns$ is given by
\begin{align}\label{lradi_W}
    W_j&:=(A-\overline{\alpha_j}I)V_j=W_{j-1}-2\Real{\alpha_j}V_j=W_0+Z_jG_j,
\end{align}
with $W_0:=B$,  
$G_j:=[\gamma_1,\ldots,\gamma_j]^*\otimes I_s\in\R^{js\times s}$, 
$\gamma_i:=\sqrt{-2\Real{\alpha_i}}$ for $i=1,\ldots,j$.
\end{theorem}
The residual factors $W_j\in\Cns$ will play a very important role in this article. 
As already indicated in line~\ref{alg:resstop} in Algorithm~\ref{alg:glradi}, the residual factorization~\eqref{lradi_res} greatly helps to cheaply compute
the norm of the residual matrix which is useful as a stopping criterion.
The low-rank solution factors $Z_j$ generated by the LR-ADI iteration
 solve certain Sylvester
equations. Similar results regarding an older version of the LR-ADI iteration can be found in~\cite{morLi00,LiW02}.
\begin{corollary}[{\cite[Lemma~3.1]{PanW13a}, 
\cite[Lemma~5.12]{morWol15}, \cite[Corollary~3.9]{Kue16}}]\label{cor:lradi_lrf-CASE}
 With same assumptions and notations as in Theorem~\ref{thm:lradiRes}, the low-rank 
factor $Z_j$
after $j$ steps of the LR-ADI iteration (Algorithm~\ref{alg:glradi})
satisfies the Sylvester equations
\begin{subequations}\label{lradi_lrf-CASE}
\begin{align}\label{lradi_lrf-CASE1}
 AZ_{j}-Z_{j}S_j&=BG_j^*,\\\label{lradi_lrf-CASEj}
 AZ_{j}+Z_{j}\overline{S}_j^*&=W_jG_j^*,
\end{align}
where 
\begin{align}
S_j:=\left[\begin{smallmatrix}
 \alpha_1&\gamma_1\gamma_2&\cdots&\gamma_1\gamma_j\\
&\ddots&\ddots&\vdots\\
&&\ddots&\gamma_{j-1}\gamma_j\\
&&&\alpha_j
\end{smallmatrix}\right]\otimes I_s\in\C^{js\times js}.
\end{align}
\end{subequations}
\end{corollary}
\begin{remark}\label{realadi}
 In practice, although~\eqref{cale} is defined by real $A,~B$, complex shift parameters can occur. We assume that the set of shifts $\lbrace
\alpha_1,\ldots,\alpha_j\rbrace$ is closed under complex
conjugation and that pairs of complex conjugated shifts occur subsequently, i.e., $\alpha_{i+1}=\overline{\alpha_i}$ if $\Imag{\alpha_i}\neq 0$. 
These complex
parameters pairs are in practice dealt within the LR-ADI iteration by a double step fashion~\cite{BenKS13,BenKS14b,Kue16} resulting in real low-rank factors
$Z_j$ and, important for this study, real low-rank residual factors $W_j$. Real version of the above results can be established but for brevity and clarity we
keep the
shorter complex versions in the remainder. The real version of the LR-ADI iteration will nevertheless be used in the numerical experiments in the end.
\end{remark}
\subsection{Shift parameters}\label{ssec:shifts}
The approximation error $X-X_j$ and residual $R_j$ can be expressed as
\begin{align*}
 X-X_j&=\cM_j(X-X_0)\cM_j^*,\quad R_j=\cM_j R_0\cM_j^*,\quad\\
\text{with} \quad\cM_j&=\prod\limits_{i=1}^j\cC(A,\alpha_i),\quad\text{and}\quad
\cC(A,\alpha):=(A-\overline{\alpha}I)(A+\alpha I)^{-1}
\end{align*}
is a Cayley transformation of  $A$.
Taking norms leads to
\begin{align*}
 \frac{\| X-X_j\|}{\|X-X_0\|}\leq c\rho(\cM_j)^2,\quad \frac{\|R_j\|}{\|R_0\|}\leq c\rho(\cM_j)^2,\quad c\geq 1.
\end{align*}
Because of $\Lambda(A)\subset\C_-$ as well as $\alpha_i\in\C_-$ it holds $\rho(\cC(A,\alpha_i))<1$, $1\leq i\leq j$ and, consequently,
$\rho(\cM_j)<1$ is getting smaller as the ADI iteration proceeds. This motivates to select the shifts $\alpha_i$ such that $\rho(\cM_j)$ is as small as possible
leading to the ADI parameter problem
\begin{align}\label{minmax}
 \min\limits_{\alpha_i,\ldots,\alpha_j\in\C_-}\left(\max\limits_{\lambda\in\Lambda(A)}\left|\prod\limits_{i=1}^j\frac{\lambda-\overline{\alpha_i}}{
\lambda+\alpha_i } \right|\right).
\end{align}
Several shift selection strategies have been developed based on~\eqref{minmax}, e.g., the often used Wachspress~\cite{Sab07,Wac13} and Penzl~\cite{Pen99} 
selection approaches, which precompute a number of shifts before the actual LR-ADI iteration. There, the spectrum $\Lambda(A)$ in~\eqref{minmax} is replaced by
an easy to compute approximation, typically using a small number of approximate eigenvalues generated by Arnoldi and inverse Arnoldi processes. The shifts are
then obtained by means of elliptic functions in the Wachspress approach~\cite{Sab07,Wac13} and, respectively, heuristically in the Penzl
approach~\cite{Pen99} .
Starting from~\eqref{minmax} for selecting shifts has, however, some shortcomings.
A disadvantage from the conceptual side is that the min--max problem~\eqref{minmax} does only take (approximate) eigenvalues of $A$ into account.
No 
information regarding the inhomogeneity $BB^*$ of the CALE~\eqref{cale} is incorporated, although the low-rank property of $BB^*$ is one 
 significant factor for the singular value decay of the solution and, hence, for the existence of low-rank approximations~\cite{AntSZ02,Gra04,TruV07}. 
Furthermore, no information regarding the eigenvectors of $A$ enters~\eqref{minmax}. While this might not be a big issue for CALEs defined by symmetric
matrices, in the nonsymmetric case the spectrum alone 
might not be enough to fully explain the singular value decay of
the solution, see, e.g., the discussions in~\cite{Sab07,BakES14}.

Because only approximate eigenvalues can be used for large-scale problems, Wachspress and Penzl shift strategies can also suffer from poor eigenvalue
estimates~\cite{Sab07} and the cardinality of the set of approximate eigenvalues (Ritz values) is an unknown
quantity the user has to provide in advance. Even tiny changes in these quantities can greatly alter the speed of the error or residual reduction in the ADI
iteration.
Because the strategies based on~\eqref{minmax} are usually in general carried out in advance, i.e., shifts are generated before the actual
iteration, no information about the current progress of the iteration is incorporated.

Here, we are interested in adaptive shift selection and generation strategies that circumvent these issues.
Our goal is that these approaches incorporate the current stage of the iteration into account, and that the shifts are generated 
automatically and in a numerically efficient way during the iteration, i.e., the shift computation should consume only a small fraction of the total numerical
effort of the LR-ADI iteration. Next, we review commonly used existing dynamic shift selection approaches and propose some enhancements. 
\subsubsection{Ritz value based dynamic shifts}\label{sssec:projshifts}
First steps regarding dynamic shift approaches were made in~\cite{BenKS14b} by using Ritz values of $A$ with respect to a subspace
$\cQ_{\ell}=\range{Q_{\ell}}\subseteq\range{Z_j}$, where $Q_{\ell}\in\R^{n\times \ell}$ has orthonormal columns. The typical choice is to select the most recent block columns of $Z_j$ for spanning $\cQ_{\ell}$:
\begin{align}\label{spaceZh}
	\cQ_{\ell}=\cZ(h):=\range{[V_{j-h+1},\ldots,V_j]}
\end{align}
 with $1\leq h \leq j$ to keep the space dimension small. 
The Ritz values are given by $\Lambda(H_{\ell})$ with $H_{\ell}:=Q_{\ell}^*AQ_{\ell}$ 
and can, e.g., be plugged into the Penzl heuristic to select $g\leq\ell$ shift parameters. 
It can happen that $\Lambda(H_{\ell})\cap\C_+\neq\emptyset$ in which
case
we simple negate all unstable Ritz values. Once these $g$ shifts have been used, 
the generation and selection process is repeated with $Z_{j+g}$.  
Despite its simplicity, this idea already led to significant speed ups of the LR-ADI iteration, in particular for
nonsymmetric problems where the a priori computed shifts resulted in a very slow convergence. Due to this success, this approach is the default shift selection
routine in the \mmess{} software package~\cite{SaaKB16-mmess-1.0.1}. 
Further details on an efficient construction of $H_{\ell}$ are given later. This basic selection strategy can be modified in the following ways.
\subsubsection{Convex hull based  shifts}\label{sssec:convshifts}
Motivated by the connection of LR-ADI to rational Krylov subspaces~\cite{morLi00,DruKS11,FlaG12,PanW13a,morWol15} 
we can borrow the greedy shift selection strategy from~\cite{DruS11} which was
developed for the rational Krylov subspace method for~\eqref{cale}. Let $\cS\subset\C_-$ be the convex hull of the set of Ritz values $\Lambda(H_{\ell})$ and
$\partial \cS$ its boundary. For a discrete subset $\cD\subset\partial \cS$ one tries to heuristically find $\alpha\in\cD$ that reduces the magnitude of
the rational function (cf.~\eqref{minmax}) connected to the previous LR-ADI steps the most.  In contrast to the Ritz value based shift selection discussed
above, the convex hull
based selection will only provide a single shift parameter to be used for the next iteration and,
thus, the selection process has to be executed in every iteration step.  
Note that this approach employed in RKSM uses the Ritz values associated to the 
full already computed rational Krylov subspace, while in LR-ADI we only use a smaller subspace~\eqref{spaceZh}. 
\subsubsection{Residual-Hamiltonian based shifts}\label{sssec:Hres}
Both strategies mentioned so far select shift parameters on the basis of the eigenvalues of a compressed version $H_{\ell}$ of $A$. A different modification developed for
the RADI method~\cite{BenB16,BenBKetal18} for algebraic Riccati equations also takes some eigenvector information into account. For Riccati equations, the
core idea is to consider a projected version of the associated Hamiltonian matrix which we can simplify for CALEs. If $[P,Q]^*$ spans the stable
$n$-dimensional invariant subspace of
\begin{align*}
 \cH_0:=\begin{bmatrix} A^*&0\\
           BB^*& -A\end{bmatrix},  
\end{align*}
then $X=PQ^{-1}$ solves~\eqref{cale}. Let $X_j\approx X$ be obtained by LR-ADI, then all later steps 
can be seen as the application of LR-ADI to the residual Lyapunov equations $A\hX+\hX A^T=-R_{j}$~\cite[Corollary~3.8]{Kue16}, where $R_{j}=W_{j}W_{j}^*$ is the
residual associated to $X_j$. The residual equations are connected to the Hamiltonian matrices 
$\cH_j:=\smb A^*&0\\
           W_jW_j^*& -A\sme$. 
Following the same motivation as in~\cite{BenBKetal18}, we set up the projected Hamiltonian $\tilde\cH_{j,\ell}:=\smb
H_{\ell}^*&0\\
           Q_{\ell}^*W_jW_j^*Q_{\ell}& -H_{\ell}\sme$, compute its stable eigenvalues $\lambda_k$ and associated eigenvectors $\smb p_k\\q_k\sme$, 
$p_k,~q_k\in\C^{\ell}$, and select
the eigenvalue $\lambda_k$ with the largest $\|q_k\|$ as next ADI shift. 
As in the convex hull based selection, this approach delivers only a single shift each time.
\section{Residual-norm minimizing shifts}\label{sec:resminshift}
In this section we discuss the main focus of this study: the shift selection strategy originally proposed in~\cite{BenKS14b}, where the objective is to find
shift parameters that explicitly minimize the
Lyapunov residual norm. Assume that step $j$ of the LR-ADI iteration has been completed and that the associated residual factor $W_j$ is a real $n\times s$
matrix (cf.~Remark~\ref{realadi}). The goal is to find a shift $\alpha_{j+1}$ for the next iteration step.
By Theorem~\ref{thm:lradiRes} it holds for the next Lyapunov residual $\|R_{j+1}\|=\|W_{j+1}\|^2$ with 
\begin{align*}
W_{j+1}&=W_{j+1}(\alpha_{j+1})=\cC(A,\alpha_{j+1})W_{j}=W_{j}-2\Real{\alpha_{j+1}}\left((A+\alpha_{j+1}I)^{-1}W_{j}
\right).
\end{align*}
Obviously, for executing step $j+1$ everything except the shift $\alpha_{j+1}$ is known. 
The parameter $\alpha_{j+1}\in\C_-$ is to be determined which reduces the Lyapunov residual norm the most from step $j$ to
$j+1$. This minimization problem can be formulated as, e.g., a complex nonlinear least squares problem (NLS)
\begin{align}\label{NLS_C}
\begin{split}
 \alpha_{j+1}&=\myargmin\limits_{\alpha\in\C_-}\half\|\Psi_j(\alpha,\overline{\alpha})\|^2,\\
 \Psi_j(\alpha,\overline{\alpha})&=\cC(A,\alpha)W_{j}=(A-\overline{\alpha}I)(A+\alpha I)^{-1}W_{j}.
 \end{split}
\end{align}
The complex function $\Psi_j(\alpha,\overline{\alpha}):\C\mapsto\C^{n\times s}$ is obviously not analytic in the complex variables $\alpha$, $\overline{\alpha}$
alone but
in the full variable $(\alpha,\overline{\alpha})$, a property typically referred to as polyanalyticity.
In the their original appearance~\cite{BenKS14b}, the residual minimizing shifts were considered via the minimization of the real valued function
$\psi_j=\|W_{j+1}\|^2$, which corresponds to the Lyapunov residual norm after using the shift parameter $\alpha$ starting from $R_{j}$. The complex
minimization
problem takes the form
\begin{align}\label{normminC}
 \alpha_j=\myargmin\limits_{\alpha_j\in\C_-}\psi^{\C}_j(\alpha,\overline{\alpha}),\quad
\psi^{\C}_j(\alpha,\overline{\alpha})&:=\|\Psi_j(\alpha,\overline{\alpha})\|^2.
\end{align}
It is clear that \eqref{NLS_C} and \eqref{normminC} essentially encode the same optimization task but differences will occur
in
the numerical treatment of both formulations. 
Splitting the complex variable into two real ones by $\alpha=\nu+\jmath \xi$ with $\nu<0$ yields that real and imaginary parts of
the next shift $\alpha_{j+1}=\nu_{j+1}+\jmath\xi_{j+1}$ can be obtained from solving
\begin{align}\label{normmin}
\begin{split}
 [\nu_{j+1},\xi_{j+1}]&=\myargmin\limits_{\nu\in\R_-,\xi\in\R}\psi_j(\nu,~\xi),\\
	 \psi_j&=\psi_j(\nu,~\xi):=\|W_{j}-2\nu\left((A+(\nu+\jmath\xi)I)^{-1}W_{j}\right)\|^2.
\end{split}
\end{align}
 Note that the objective function $\psi$ in~\eqref{normmin} always maps real variables to real values, whereas the function $\Psi$ in the NLS
formulation~\eqref{NLS_C} is only real for $\alpha\in\R$. This would still be the case if $\alpha$ in~\eqref{NLS_C} was decomposed in $\nu,~\xi$ as we did
for~\eqref{normmin}.
Since $\|X\|_2=\lambda_{\max}(X^*X)$, the minimization problems~\eqref{normminC} and \eqref{normmin} can also be understood
as eigenvalue optimization problems if $s>1$.

Naturally, if one knows that real shift parameters are sufficient, e.g. when $A=A^*$, the above minimization problems simplify in
the obvious manner by restricting the optimization to $\R_-$.
For achieving reduction of the residual as well as avoiding the singularities at $-\Lambda(A)\subset\C_+$, the constraint  $\nu<0$ is
mandatory so that tools
from constrained optimization are required. Originally, an unconstrained version of~\eqref{normmin} and derivate-free methods were used
to find a minimum of $\psi_j$~\cite{BenKS14b}. This strategy turned out to be unreliable, in particular, because unusable shifts ($\nu\geq0$) were frequently
generated.
In this article, we employ constrained, derivative based optimization approaches using the complex nonlinear least squares~\eqref{NLS_C} and the
real
valued minimization problem~\eqref{normmin}. The underlying objective functions are generally not convex and have potentially more than one minimum in $\C_-$.
Here, we will
pursue only the detection of local minima because any parameter $\alpha\in\C_-$ will yield at least some reduction of the CALE residual norm, such that the
substantially
larger numerical effort to compute global minima will hardly pay off.
The next subsection gives the structure of the required derivatives
of $\Psi_j(\alpha,\overline{\alpha})$, $\psi^{\R}_j$. Afterwards, numerical aspects such as approximating the objective
functions, solving the
minimization or least squares problems, and implementing the proposed shift generation framework efficiently in Algorithm~\ref{alg:glradi} are discussed. 
 \subsection{Derivatives of the objective functions}
For the least-squares problem~\eqref{NLS_C}, the Jacobian and conjugate Jacobian~\cite{SorVL12} of $\Psi_j(\alpha,\overline{\alpha})$ are 
\begin{align}\label{deriv_W}
\begin{split}
\frac{\partial \Psi_j(\alpha,\overline{\alpha})}{\alpha}&=-(A-\overline{\alpha} I)(A+\alpha
I)^{-2}W_{j}=\cC(A,\alpha)(A+\alpha I)^{-1}W_{j},\\
 \frac{\partial
\Psi_j(\alpha,\overline{\alpha})}{\overline{\alpha}}&=-(A+\alpha I)^{-1}W_{j}.
\end{split}
\end{align}
The structure of the derivatives for $\psi_j$ in~\eqref{normmin} is more complicated.
 \begin{theorem}[{Gradient and Hessian of the objective function~\eqref{normmin}}]\label{thm:normfun_cale_grad}
Let $\alpha=\nu+\jmath \xi\in\C_-$, $W=W_j\in\Rns$, and define
$L(\nu,\xi):=A+\alpha I$, 
\begin{align*}
S^{(i)}&:=(L(\nu,\xi)^{-1})^{i}W,\quad 
&W^{(i)}_{\alpha}&:=S^{(i)}-2\nu S^{(i+1)},\quad 
\hW^{(i)}&:=S^{(i)}-\nu S^{(i+1)},&\\
\tR_\nu:&=-(W_{\alpha}^{(0)})^*\hW^{(1)},\quad &\tR_\xi&:=(W_{\alpha}^{(0)})^*S^{(2)}&&
\end{align*}
for $i=0,\ldots,3$. 
Assume $(W^{(0)}_{\alpha})^*W^{(0)}_{\alpha}$ has $s$ distinct eigenvalues $\theta_1>\ldots>\theta_s>0$ and let
 $(\theta_{\ell}~,u_{\ell})=(\theta_{\ell}(\nu,~\xi)~,u_{\ell}(\nu,~\xi))$ with $\|u_{\ell}\|=1$, 
$\ell=1,\ldots,s$ be its eigenpairs.
Then, gradient and Hessian of~\eqref{normmin} are given by 
\begin{align}\label{calenormfun_grad}
\grad \psi_j&=4\smb
\Real{u_1^*\left((W^{(0)}_{\alpha})^*\hW^{(1)}\right)u_1}\\                    
-\nu\Imag{u_1^*\left((W^{(0)}_{\alpha})^*S^{(2)}\right)u_1}
\sme
 =
4\smb
\Real{u_1^*\tR_{\nu}u_1}\\
-\nu\Imag{u_1^*\tR_{\xi}u_1}
\sme
\end{align}
and
\begin{subequations}\label{calenormfun_Hessian}
\begin{align}
\grad^2 \psi_j=&8\smb
\Real{u_1^*\left((\hW^{(2)})^*W^{(0)}_{\alpha}+(\hW^{(1)})^*\hW^{(1)}\right)u_1}
\qquad\qquad
h_{12}
\\
h_{12}
\qquad\qquad
\nu\Real{u_1^*((S^{(3)})^*W^{(0)}_{\alpha}+\nu (S^{(2)})^*S^{(2)})u_1}
\sme\\
&+\sum\limits_{k=2}^s\tfrac{8}{\theta_1-\theta_k}
\smb
|u_1^*(\tR^*_\nu+\tR_\nu)u_k|^2&
\tilde h^{(k)}_{12}
\\
\tilde h^{(k)}_{12}
& 
|u_1^*(\tR^*_\xi-\tR_\xi)u_k|^2
\sme\\\nonumber
\text{with}\quad h_{12}&:=\half \Imag{u_1^*\left((W_{\alpha}^{(2)})^*W^{(0)}_{\alpha}-2\nu (S^{(2)})^*W^{(1)}_{\alpha}\right)u_1},\\\nonumber
\tilde h^{(k)}_{12}&:=-\Real{(u_1^*(\tR^*_\nu+\tR_\nu)u_k)(\jmath\nu
u_k^*(\tR^*_\xi-\tR_\xi)u_1)}.
\end{align} 
\end{subequations}
\end{theorem}
\begin{proof}
The results are obtained by building the partial derivatives of 
 $\hat\cC(A,\alpha=\nu+\jmath \xi)$ and of $\psi_j(\nu,~\xi)=\sigma^2_{\max}\left(
\cC(A,\nu+\jmath\xi)W\right)=\lambda_{\max}\left(\Psi^*\Psi\right)=\lambda_{\max}\left((W^{(0)}_{\alpha})^*W^{(0)}_{\alpha}\right)$ using results on derivatives
of eigenvalues of parameter dependent
matrices, e.g.,~\cite{Lan64,Men13}.
For detailed proof the reader is referred to~\cite[Section~5]{Kue16}, where also the formulas adapted to~\eqref{gcale} are given.
\end{proof}
 \subsection{Approximating the objective functions}\label{ssec:objapprox}
 The main issue arising when solving the optimization problems~\eqref{NLS_C},~\eqref{normmin} is that each evaluation of the objective functions $\Psi_j$,
$\psi_j$  at
a value
$\alpha$ requires one linear solve
with $A+\alpha I$. Moreover, the derivative formula reveal that evaluating derivatives
requires, depending on the order of the derivatives and on the used optimization problem, at least one
additional linear solve.
Thus, each evaluation within a derivate-based optimization method will be more expensive than a single LR-ADI iteration step, 
making the numerical solution of \eqref{NLS_C},~\eqref{normmin} very costly, regardless of the employed optimization algorithm, and, consequently, the shift
generation would be prohibitively expensive. 

As main contribution of this paper, this section proposes strategies to work with 
approximations of the objective functions $\tilde\Psi_j\approx \Psi_j$, $\tilde\psi_j\approx \psi_j$ which can be much cheaper evaluated
and whose construction is numerically efficient as well. 
The resulting shifts are generated by using $\tilde\psi_j$ in~\eqref{normmin} and $\tilde\Psi_j$ in~\eqref{NLS_C}. 
Our main approach is based on a projection framework using a low-dimensional subspace $\cQ\subset\Cn$, dim$(\cQ)=\ell\ll n$.
Let the columns of $Q_{\ell}\in\C^{n\times \ell}$ be an orthonormal basis of $\cQ$. We employ the usual Galerkin approach to obtain an approximation
\begin{align*}
 \Psi_j(\alpha,\overline{\alpha})&=\cC(A,\alpha)W_j\approx Q_{\ell}\cC(H_{\ell},\alpha)\tW_{\ell,j}=:\tilde\Psi_j(\alpha,\overline{\alpha}),\\
H_{\ell}&:=Q_{\ell}^*AQ_{\ell}\in\C^{\ell\times\ell},\quad
\tW_{\ell,j}:=Q_{\ell}^*W_{j}\in\C^{\ell\times s}.
\end{align*}
Because of the orthogonality of $Q_{\ell}$ it suffices to use the projected objective functions $\tilde\psi_j:=\|\cC(H_{\ell},\alpha)\tW_{\ell,j}\|_2$ and
$\hat\Psi_j(\alpha,\overline{\alpha}):=\cC(H_{\ell},\alpha)\tW_{\ell,j}$.
Evaluations of the functions and their derivatives is cheaper because the small dimension of $H_{\ell}$ allows easier to solve systems with
$H_{\ell}+\alpha I_{\ell}$.

In the following we discuss some choices for the projection subspace $\cQ$. Our emphasis is that
quantities already generated by the LR-ADI iteration are used as much as possible. Since~\eqref{NLS_C},~\eqref{normmin} have to be solved in each iteration 
step of Algorithm~\ref{alg:glradi} using a different residual factor $W_{j}$ each time, we also discuss the reuse of approximation data from step $j$ to $j+1$. 
\subsubsection{Using subspaces spanned by the low-rank factor}\label{sssec:approx_lrf}
In~\cite{Kue16} it is suggested to augment the Ritz value based shifts (Section~\ref{sssec:projshifts}) by the optimization problem~\eqref{normmin} using
$\tilde\psi_j$, 
i.e., after step $j$, the space $\cQ=\cZ(h)$ spanned by the last $1\leq h\leq j$ block columns of the already generated low-rank solution factor $Z_j=[V_1,\ldots,V_j]$ is
selected as in~\eqref{spaceZh}.
The reduced objective function $\tilde\psi_j$ is then built by $H_{\ell}$ and $\tW_{\ell,j}$.
The restriction $H_{j}$ of $A$ can be build without additional multiplications with $A$
because of~\eqref{lradi_lrf-CASE}. Let $R_j\in\C^{hs\times hs}$ so that $Q_j=[V_{j-h+1},\ldots,V_j]R_j$ has orthonormal columns. Then
\begin{align}\label{restrHjZj}
 H_{j}:=Q_j^*AQ_j=Q_j^*W_jG_{j,h}^*R_j-R_j^{-1}S_{j,h}^*R_j, 
\end{align}
where $G_{j,h}$, $S_{j,h}$ indicate the last $h$ block rows (and columns) of $G_{j}$, $S_{j}$ from~\eqref{lradi_lrf-CASE}.
Even though this space selection is rather intuitive, it led to impressive results often outperforming existing shift selecting
strategies~in~\cite{Kue16}. The obtained rate of the residual norm reduction in the LR-ADI iteration was very close to the case
when the the exact objection function was used in~\eqref{normmin}, indicating a sufficiently good approximation of $\psi_j$ at low generation costs. Note that,
the concept of approximating an expensive to evaluate objective function by projections onto already built up subspaces can also be found for other problems,
e.g., in the context of model order reduction~\cite{morCasPL17}.
\subsubsection{Krylov and extended Krylov subspace based approximations}\label{sssec:approx_Krylov}
Consider the block Krylov subspace of order $p$ as projection space:
\begin{align*}
 \cQ=\cK_p(A,W_j):=\myspan{W_j,AW_j,\ldots,A^{p-1}W_j}.
 \end{align*}
This is a common strategy for approximating the product of a parameter independent, large-scale matrix function times a block vector $f(A)W_j$, see,
e.g.,~\cite{Saa92,Kni92,FroS08,FroLS17}. 
For our parameter dependent matrix function $\cC(A,\alpha)$, this choice can be motivated by considering the boundary of the stability region, where
$\cC(A,0)W_j=W_j$ which is the first basis block of $\cK_p(A,W_j)$. On the other hand, at $\alpha=0$ we have
for the derivatives, e.g.,
\begin{align*}
  \frac{\partial \Psi_j(\alpha,\overline{\alpha})}{\alpha}&=-A^{-1}W_{j}
\end{align*}
and, moreover, $\grad^2 \psi_j$ at $\alpha=0$ involves expressions with $A^{-2}W_j$. In order to get, at least near the origin, a good approximation of
$\Psi_j$, $\psi_j$ and
their derivatives, this motivates to also incorporate information from a low-order inverse Krylov subspace $\cK_m(A^{-1},A^{-1}W_j)$ to the projection space
$\cQ$.
Hence, we consider the extended Krylov subspace
\begin{align*}
 \cQ=\cE_{p,m}(A,W_j)&:=\cK_p(A,W_j)\cup\cK_m(A^{-1},A^{-1}W_j)\\
 &=\myspan{W_j,AW_j,\ldots,A^{p-1}W_j,A^{-1}W_j,\ldots,A^{-m}W_j}
\end{align*}
as projection subspace.
Constructing the basis matrix $Q_{p,m}$ and the restrictions $H$, $\tW_{j}$ can
be done efficiently by the extended Arnoldi process~\cite{Sim07}, requiring essentially only matrix vector products and linear solves with $A$. 
However, for approximating $\Psi_j$, $\psi_j$, the right hand side $W_j$ changes throughout the ADI iteration, which would necessitate to construct a new
orthonormal basis
associated to $\cE_{p,m}(A,W_j)$ in each LR-ADI iteration step. 
As an auxiliary contribution, the next theorem shows that this is not needed for $j>1$ and shows how the subspaces $\cE_{p,m}(A,W_j)$ evolve from an initial 
subspace $\cE_{p,m}(A,B)=\cE_{p,m}(A,W_0)$. Note that because of the arising block matrices $q_i\in\Cns$,
$i=1,\ldots,\ell$, expressions $\myspan{q_1,\ldots, q_{\ell}}$ and
$\range{[q_1,\ldots,q_{\ell}]}$ in the theorem are to be understood in the block-wise sense following the framework defined 
in~\cite{FroLS17}. In particular,  
$\myspan{q_1,\ldots, q_{\ell}}=\lbrace\sum\limits_{i=1}^{\ell}q_i\Xi_i,~\Xi_i\in\C^{s\times s}\rbrace$ ,
similarly for $\range{\cdot}$. 
 \begin{theorem}\label{thm:krylovresvec}
 For $j>1$ and $0\leq p,m\leq n$ (with at least one of the orders $p,m$ nonzero) it holds
 \begin{align*}
  \cE_{p,m}(A,W_j)\subseteq\cE_{p,m}(A,B)\cup\range{Z_j}.
 \end{align*}
\end{theorem}
\begin{proof}
For simplicity and clarity, we restrict the proof to the case $p>0,~m=0$. The more general situation can be elaborated similarly.
Let $\cK_p(A,B)=\range{K_p(A,B)}$, where 
$K_p(A,B):=[B,AB,\ldots,A^{p-1}B]\in\R^{n\times ps}$ is the associated block Krylov 
matrix. Likewise, $K_p(A,W_j)$ is the Krylov matrix w.r.t. $\cK_p(A,W_j)$.

We show $\myspan{A^{p-1}W_j}\subset\cK_p(A,B)\cup\range{Z_j}$
via induction. For $p=1$, it holds 
$A^0W_j=W_j=K_1(A,B)I_s+Z_jS_j^0G_j=B+Z_jG_j$ because of~\eqref{lradi_W}. Let the claim be true for all powers up to $p-2$,
i.e.,
it holds
\begin{align*}
 A^{p-2}W_j=K_{p-1}(A,B)M_{p-1}+Z_jN_{p-2}
\end{align*}
for some matrices $M_{p-1}\in\R^{s(p-1)\times s}$, $N_{p-2}\in\R^{js\times s}$ of rank $s$.
By using~\eqref{lradi_W} and~\eqref{lradi_lrf-CASE1}, we obtain for 
the induction step from matrix power $p-2$ to $p-1$
\begin{align*}
 A^{p-1}W_j&=A(A^{p-2}W_j)=A(K_{p-1}(A,B)M_{p-1}+Z_jN_{p-2})\\
  &=AK_{p-1}(A,B)M_{p-1}+BG_j^*N_{p-1}+Z_{j}S_jN_{p-2}\\
 &=[B,AK_{p-1}(A,B)]\begin{bmatrix}
               G_j^*N_{p-2}\\
               M_{p-1}
              \end{bmatrix}+Z_jS_jN_{p-2}.
\end{align*}
It is easy to see that $N_{p-2}=S_j^{p-2}G_j$
which establishes for $p>1$
\begin{align}\label{KrylovvecW}
 A^{p-1}W_j=K_p(A,B)M_k
+Z_j(S_j)^{p-1}G_j,\quad 
M_p:=\begin{bmatrix}
                 G_j^*S_j^{p-2}G_j\\
                 M_{p-1}
\end{bmatrix},~M_1:=I_s,
\end{align}
proving the assertion. For $m>0$ we have $A^{-1}Z_j=Z_jS_j^{-1}-A^{-1}BG_j^*S_j^{-1}$ by~\eqref{lradi_lrf-CASE1} which leads immediately to
$A^{-1}W_j=A^{-1}B(I_s-G_j^*S_j^{-1}G_j)+Z_jS_j^{-1}$ and, consequently,  $\cK_m(A^{-1},A^{-1}W_j)\subseteq\cK_m(A^{-1},A^{-1}B)\cup\range{Z_j}$ can be shown as
for the standard Krylov subspace. The unification yields the claim for $\cE_{p,m}$. 
\end{proof}

The consequence of Theorem~\ref{thm:krylovresvec} is that for every iteration 
step $j>1$, a basis for the subspace $\cE_{p,m}(A,W_j)$ can be constructed from the initial basis for
$\cE_{p,m}(A,B)$ and the low-rank factor $Z_j$.
By concatenating the block columns $W_j,AW_j,\ldots A_{p-1}W_j$ from~\eqref{KrylovvecW} we obtain
\begin{align}\label{KrylovmatW}
  K_p(A,W_j)&=K_p(A,B)T_p^{\cK}+Z_jK_p(S_j,G_j),\\\nonumber
 T_p^{\cK}&:=\smb I_s&T_2&\cdots&T_p\\
   &\ddots&\ddots&\vdots\\
   &&\ddots&T_2\\
   &&&I_s\sme
 \in\C^{sp\times sp}\quad\text{with}\quad \begin{matrix}
               T_i=G_j^*S^{i-2}_jG_j\in\C^{s\times s},\\
               1<i\leq p,
\end{matrix}
\end{align}
and for $m>0$ a straightforward generalized expression can be found.
Of course, from a numerical point of view it is not wise to work with the explicit (extended) 
Krylov matrices or the matrix $T_p^{\cK}$. 
Instead, we propose to
use
\begin{align}\label{Krylov:buildbasis}
 Q_{p,m}(A,W_j)=\texttt{orth}[Q_{p,m}(A,B),\omega_j],\quad \omega_j:=Z_jQ_{p,m}(S_j,G_j)
\end{align}
as projection space, where $\texttt{orth}$ refers to any stable orthogonalization routine. There, $Q_{p,m}(A,W_j)$, $Q_{p,m}(A,B)$, and $Q_{p,m}(S_j,G_j)$
are
orthonormal basis matrices for the extended Krylov spaces $\cE_{p,m}(A,W_j)$, $\cE_{p,m}(A,B)$, and $\cE_{p,m}(S_j,G_j)$, respectively.
The basis matrix $Q_{p,m}(A,B)\in\R^{n\times (p+m)s}$ can be constructed by an extended Arnoldi process,
which is only required once before the actual ADI iteration. Constructing the basis matrix $Q_{p,m}(S_j,G_j)\in\C^{js\times (p+m)s}$ only requires
working with $j\ll n$ dimensional data.
More details on the numerical implementation are given later in Section~\ref{sssec:implementation}.
\begin{remark}
 \begin{enumerate}
 \item The result for $m=0$ indicates a basic
framework for acquiring a basis of $\cK_p(A,W_j)$ from $\cK_p(A,B)$ without new matrix vector products involving $A$ and, thus, it could be useful for
iteratively solving the shifted linear system in LR-ADI by Krylov subspace methods. Since this is beyond the scope of
this study, we leave exploiting Theorem~\ref{thm:krylovresvec} for iterative linear solves for future work. 
 \item The motivation for using $\cE_{p,m}$ was to improve the approximation of $\Psi_j$, $\psi_j$ near the origin. However, in practice 
 the origin will be excluded in the actual optimization. One can use shifted spaces defined by $A-\phi I$, $\phi>0$, e.g., if one can expect that the local
minima have $\Real{\alpha}<-\phi<0$. This only changes the inverse Krylov subspace $\cK_m((A-\phi I)^{-1},(A-\phi I)^{-1}B)$ since the standard Krylov
subspaces are shift-invariant. 
  \item A intuitive extension would be the use of rational Krylov subspaces 
\begin{align*}
 \cK^{\text{rat}}_r(A,W_j,\boldsymbol{\beta})=\myspan{(A+\beta_1I)^{-1}W_j,\ldots,\prod\limits_{i=1}^r(A+\beta_iI)^{-1}W_j}
\end{align*}
with poles $\boldsymbol{\beta}=\lbrace
\beta_1,\ldots,\beta_r\rbrace$. The motivation for this choice is to approximate  $\Psi_j,~\psi_j$ at
values $\beta_i$, $1\geq i\geq r$ that may lie in the interior of the optimization region. However, this requires knowledge of adequate shifts $\beta_i$ such
that the relevant behavior of $\Psi_j,~\psi_j$ is captured.
Exactly this makes
the rational Krylov approximation problematic, since it is currently not known where in $\C_-$ the local minima of $\Psi_j,~\psi_j$ are located and where
suitable $\beta_i$
should be placed. Therefore, we will restrict in the remainder to the standard and extended Krylov subspace approaches.
The basis construction for $\cK^{\text{rat}}_r(A,W_j,\boldsymbol{\beta})$ can be build similarly from an a priori generated basis of
$\cK^{\text{rat}}_r(A,B,\boldsymbol{\beta})$ obtained by a rational (block) Arnoldi process~\cite{Ruh94c,Gue13}.
If $\beta_i=\alpha_i$, $1\leq i\leq r\leq j$, it is well known that $\cK^{\text{rat}}_r(A,B,\boldsymbol{\beta})\subseteq\range{Z_j}$
\cite{LiW02,FlaG12,PanW13a,morWol15}. However, even in this case $\myspan{(A+\alpha_iI)^{-1}W_j}\varsubsetneq\range{Z_j}$ such that the construction mentioned
in Section~\ref{sssec:approx_lrf} is not a true rational Krylov approximation.
 \end{enumerate}
\end{remark}
\subsubsection{Implementation}\label{sssec:implementation}
Before the approaches for solving the optimization problems are investigated, the numerical implementation of the proposed strategy for approximating the
objective function $\Psi_j$, $\psi_j$ within the LR-ADI iteration is discussed. 
Assume that before the LR-ADI iteration is started, a block extended Arnoldi process~\cite{Sim07} with orders $p,m$ is applied to $A,B$ which provides 
\begin{align*}
Q_B^*Q_B=I,\quad\range{Q_B}=\cE\cK_{p,m}(A,B),\quad P_B&:=AQ_B,  
\end{align*}
$Q_{B}(1:s,:)\eta=B,~\eta\in\R^{s\times s}$, and $H_{B}=Q_B^*AQ_B=Q_B^*P_B\in\R^{(p+m)s\times (p+m)s}$, i.e., the restriction of $A$ w.r.t.
$\cE\cK_{p,m}(A,B)$.
For later use, $Q_B,~P_B$, and $H_B$ are stored. If no shift parameters for the first LR-ADI steps are provided, a residual-norm minimizing shift $\alpha_1$ is
computed by solving
a reduced optimization problem~\eqref{normmin} defined by $H_B$ and $\tW_0:=Q_B^*B=[\eta^*,0,\ldots,0]^*\in\R^{(p+m)s\times s}$.  
Suppose $j$ steps of the LR-ADI iterations have been carried out and we look for a shift $\alpha_{j+1}$ for the next step by using a reduced objective function
constructed from the approximation space $\cK_p(A,W_j)$.
Motivated by Theorem~\ref{thm:krylovresvec}, the augmented basis matrix~\eqref{KrylovmatW} w.r.t. the augmented space $\cE\cK_{p,m}(A,B)\cup \myspan{\omega_j}$,
where $\omega_j:=Z_jQ_{S_j}\in\C^{n\times (p+m)s}$ and $Q_{S_j}\in\C^{js\times (p+m)s}$ is the orthogonal basis matrix spanning $\cE\cK_{p,m}(S_j,G_j)$.
Executing the extended Arnoldi process with $S_j,~G_j$ 
is extraordinarily cheap because it involves only quantities of dimension $j$ due to the Kronecker structure of $S_j, G_j$
(cf.~\eqref{lradi_W},\eqref{lradi_lrf-CASE1}). The matrix $\omega_j$ gives the relevant part of $\range{Z_j}$ needed for 
$\cE\cK_{p,m}(A,W_j)$. If $j\leq p+m$, $Z_j$ has less than or exactly $(p+m)s$ columns, and $w_j:=Z_j$ is used as simplification. 
To orthogonally extend the basis of $\cE\cK_{p,m}(A,B)$ by $\myspan{\omega_j}$, we may employ any stable orthogonalization routine, e.g., an iterative
block Gram Schmidt process. Consider for illustration one sweep of block Gram Schmidt
\begin{align*}
 h_j:=Q_B^*w_j\in\C^{(p+m)s\times (p+m)s},\quad \hat\omega_j:=\omega_j-Q_Bh_j,\quad Q_{Z_j}:=\hat\omega_j\hat h_j,
\end{align*}
where $\hat h_j$ orthonormalizes the columns of $\hat\omega_j$ and is obtained by a thin QR decomposition of $\hat\omega_j$. Hence,
$Q_j:=[Q_B,Q_{Z_j}]\in\C^{n\times 2(p+m)s}$ is the sought orthogonal basis matrix, $\tW_j:=Q_j^*W_j$, and
\begin{subequations}\label{updHj}
\begin{align}
 H_j:=Q_j^*AQ_j
 =\begin{bmatrix}\smb
 H_B\\
 (Q_{Z_j})^*P_B
 \sme&
 Q_j^*P_{Z_j}
 \end{bmatrix}\in\C^{2ps\times 2ps},\quad P_{Z_j}:=AQ_{Z_j}.
\end{align}
The additional $p+m$ matrix vector products with $A$ can be avoided by constructing $P_{Z_j}$ in conjunction with the Gram-Schmidt orthogonalization of
$\omega_j$ against $Q_B$ and using
\eqref{lradi_lrf-CASE1}:
\begin{align}
 \hP&:=A\omega_j=AZ_jQ_{S_j}=BG_j^*Q_{S_j}+Z_jS_jQ_{S_j},\\
 P_{Z_j}&=AQ_{Z_j}=A\hat\omega_j\hat h_j=(A\omega_j-AQ_Bh_j)\hat h_j=(\hP-P_Bh_j)\hat h_j.
\end{align}
\end{subequations}

Unless $A+A^*\prec 0$, it can happen that the restriction $H_j$ has unstable eigenvalues which would be problematic for the usage of the compressed objective functions.
As a basic counter measure, we replace $H_j$ by its Schur form $H_j\leftarrow Q_{j,H}^*H_jQ_{j,H}$ and simply negate any arising unstable eigenvalues that appear on the diagonal of $H_{j}$. 
Transforming the compressed objective function into the Schur basis $Q_{j,H}$ also simplifies the evaluation of function and derivatives due to the
(quasi)triangular structure of the Schur form.

The generation of the shift $\alpha_{j+1}$ for the next LR-ADI step $j+1$ is summarized in
Algorithm~\ref{alg:resmin} including both projection subspace choices from Sections~\ref{sssec:approx_lrf} and~\ref{sssec:approx_Krylov}. 
\begin{algorithm}[t]
\SetEndCharOfAlgoLine{}
\SetKwInOut{Input}{Input}\SetKwInOut{Output}{Output}
  \caption{Construction and solution of reduced minimization problems.}
  \label{alg:resmin}
    \Input{LR-ADI iteration index $j$, low-rank solution factor $Z_j$, residual factor $W_j$, previously used shifts $\lbrace\alpha_1,\ldots,\alpha_j\rbrace$,
orders
$p,m$ for extended Krylov
subspace, matrices $Q_B,~H_B=Q_B^*(AQ_B)$ of initial space $\cE\cK_{p,m}(A,B)$, number $h>0$ of previous block columns of $Z_j$ if $p=m=0$.}
    \Output{Next shift $\alpha_{j+1}$ for LR-ADI iteration}
  \eIf{$j>1$}
  {%
  \eIf{$p>0$ and $m>0$}{%
  \eIf{$j\leq p+m$}{%
  Set $Q_{S_j}=1$.\;}
  {%
  Generate orthonormal basis $Q_{S_j}\in\C^{sj\times (p+m)s}$ for $\cE_{p,m}(S_j,G_j)$ with $S_j,~G_j$ from~\eqref{lradi_W},~\eqref{lradi_lrf-CASE}.\;
}
$Q_j=\texttt{orth}[Q_B,Z_jQ_{S_j}]$, . 
}
{
$Q_j=\texttt{orth}[Z_j(:,(j-\min(j,h))s+1:js)]$.\;
} 
$H_j=Q_j^*(AQ_j)$ (using~\eqref{restrHjZj} or~\eqref{updHj}), $\tW_{j}:=Q_j^*W_{j}$.
  }  
  {$H_j=H_B$, $\tW_{j}=Q_B^*W_0(=Q_B^*B)$.\;}
  Compute Schur form $H_j\leftarrow Q_{j,H}^*H_jQ_{j,H}$ (negate unstable eigenvalues on demand), $\tW_{j}\leftarrow Q_{j,H}^*\tW_j$\;
  Find local minimizer $\alpha_{j+1}=\nu+\jmath \xi\in\C_-$ by solving compressed optimization problems~\eqref{NLS_C},~\eqref{normmin} defined by
  $H_j,~\tW_j$.\;
\end{algorithm}
\paragraph{Dealing with Generalized Lyapunov Equations}
In practice often generalized Lyapunov equations~\eqref{gcale} arise with an additional, invertible matrix $M\in\Rnn$. The LR-ADI iteration for~\eqref{gcale} is
given by
\begin{align}\label{lradi_M}
	V_j=(A+\alpha_j M)^{-1}W_{j-1},\quad W_j=W_{j-1}+\gamma_j^2MV_j,\quad W_0:=B,
\end{align}
(see, e.g.~\cite{morBenKS13,Kue16}) leading to generalizations of the objective functions
\begin{align*}
 \Psi^M_j(\alpha,\overline{\alpha})=(A-\overline{\alpha}M)(A+\alpha M)^{-1}W_{j},\quad
\psi_j^M(\alpha,\overline{\alpha})&=\|\Psi^M_j(\alpha,\overline{\alpha})\|^2.
\end{align*}
Approximating the generalized objective functions by using subspaces of $\range{Z_j}$ as in Section~\ref{sssec:approx_lrf} leads to
$\tilde\Psi^M_j\approx\Psi^M_j$ defined by
\begin{align*}
N_j:=Q_j^*MQ_j,\quad H_{j}:=Q_j^*AQ_j=Q_j^*W_jG_{j,h}^*R_j-N_jR_j^{-1}S_{j,h}^*R_j,\quad \tW_{j}:=Q_{j}^*W_{j},
\end{align*}
where $Q_j$,~$R_j$ come from a thin QR-factorization of the $h$ newest block columns of $Z_j$.

\smallskip
For the (extended) Krylov
subspace approximations proposed in Section~\ref{sssec:approx_Krylov}, minor complications arise because these subspaces are only defined by a single
$n\times n$
matrix. This can be dealt with by defining, e.g.,  $A_M:=M^{-1}A$, $W_{M,j}:=M^{-1}W_j$ and using
\begin{align*}
 \Psi^M_j(\alpha,\overline{\alpha})&=M\mathring{\Psi}_j(\alpha,\overline{\alpha}),\quad
\mathring{\Psi}_j(\alpha,\overline{\alpha})=\mathring{\Psi}_j:=(A_M-\overline{\alpha}I)(A_M+\alpha I)^{-1}W_{M,j},\\
 \psi_j^M(\alpha,\overline{\alpha})&=\|M\mathring{\Psi}_j\|^2=\lambda_{\max}(\mathring{\Psi}_j^*M^*M\mathring{\Psi}_j).
\end{align*}
The objective function approximation framework presented before can still be used except that $Q_{B}$ now spans $\cE\cK_{m,p}(A_M,~B_M)$ for $B_M:=M^{-1}B$.
For~\eqref{lradi_M}  the relations~\eqref{lradi_lrf-CASE} hold for $A_M,~B_M,~W_{M,j}$ such that we can orthogonally augment $Q_B$ by $Q_{Z_j}$ exactly as
before to $Q_j=[Q_B,~Q_{Z_j}]$ and use the approximations
\begin{align*}
 \Psi^M_j\approx F_{M,j}\mathring{\hat\Psi}_j(\alpha,\overline{\alpha}),\quad
F_{M,j}:=MQ_{j},\quad \mathring{\hat\Psi}:=\cC(H_{j},\alpha)Q_{j}^*W_{M,j}.
\end{align*}
The matrix $F_{M,j}\in\R^{n\times 2(m+p)s}$ is independent on the optimization variables and can therefore be easily integrated into the compressed
optimization problems via, e.g., a thin QR factorization $F_{M,j}=Q_{M,j}R_{M,j}$.
\subsection{Solving the Optimization Problem}\label{ssec:solveopti}
Having constructed the reduced objective function $\tilde\Psi_j$, $\tilde\psi_j$ by the approaches discussed before,
we plan to find a local minimizer with a derivative-based numerical optimization routine. Here, we omit most details on the 
optimization routines as more information can be found in the given citations and references therein as well as standard literature on numerical
optimization~\cite{NocW99}.

The constraints in~\eqref{NLS_C}--\eqref{normmin} can for practical purposes be given by
\begin{align}\label{GCALE_opti_bounds}
\nu_-\leq\nu\leq\nu_+,\quad 0\leq\xi\leq\xi_+,\quad -\infty<\nu_-<\nu_+<0,\quad \xi\in\R_+,
\end{align}
where the imaginary part was restricted to the nonnegative real numbers because we exclusively consider CALEs defined by 
real matrices and the generated set of shift parameters is supposed to be closed under complex conjugation.
We set $\nu_\pm,~\xi_+$ by approximate spectral data of $A$ using the extremal eigenvalues of $H_j$. 
\smallskip
Often, optimization algorithms require an initial guess to start, and the first value returned by any of the projection based shift selection approaches
(Section~\ref{sssec:projshifts}) can be employed
for this. We use the shift obtained by the Residual-Hamiltonian approach as initial guess since this led to the most promising results.\\
\smallskip
For solving the polyanalytic, nonlinear least square problems~\eqref{NLS_C} with constraints~\eqref{GCALE_opti_bounds} we use the routine
\texttt{nlsb\_gndl} from the \tensorlab{} software
package~\cite{tensorlab30}. The routine
\texttt{nlsb\_gndl} is based on a projected Gauss-Newton type method. 
In principle, the functionality of \tensorlab{} would also allow to solve the complex minimization problem~\eqref{normminC}.
Only if~\eqref{NLS_C} is restricted to real variables
$\alpha\in\R_-$, the routine~\texttt{lsqnonlin} of the \matlab~optimization toolbox\texttrademark{} is an alternative.
For solving the real valued constrained minimization problem~\eqref{normmin} in real variables, a large variety of software solutions is available.
The 
\matlab{} optimization toolbox\texttrademark{} provides with the \texttt{fmincon} routine a general purpose optimizer, which comes with different choices
for the internal optimization algorithms, e.g., interior-point~\cite{WalMNetal06} and trust-region-reflective algorithms~\cite{ColL96}, which both
allow to specify hard-coded Hessians via the explicit formulas~\eqref{calenormfun_Hessian}.\\
However, when $s>1$ the assumption in
Theorem~\ref{thm:normfun_cale_grad} that the parameter
dependent matrix $(W^{(0)}_{\alpha})^*W^{(0)}_{\alpha}$ has $s$ distinct eigenvalues 
$\theta_i(\nu,\xi)$, $1\leq i\leq s$ can in practice be violated. For instance, it can happen that $\theta_1(\nu,\xi)$ and
$\theta_2(\nu,\xi)$ coalesce at certain points $\nu,\xi$. Consequently, at these points the derivatives of $\psi_j$ do not exist, see, e.g.,
\cite{Ove92,Lew97,LewO13,Men13,MenYK14,CurMo17}. Practical observations show that especially minima of $\psi_j$ are often attained at those points.
The same problems are also present for the reduced optimization problem with $\tilde\psi_j$.
Obviously, if $s=1$, such issues are
not present which motivated the simple approach in~\cite{Kue16} to prevent these instances by simply transforming the eigenvalue optimization to a
scalar
optimization problem.
This can be achieved by multiplying the compressed residual factors $\tW_{j}$ with an appropriate tangential vector $\hW_{j}=\tW_{j}t$, where an
obvious choice for $t$ is the left singular vector corresponding to
the largest singular value of $\tW_{j}$. The associated modified objective function is then 
\begin{align}\label{normmin_t}
 \hW_{j}=\tW_{j}t,\quad \hat\psi_j:=\|\hW_{j}-2\nu(H_j+(\nu+\jmath \xi)I)^{-1}\hW_{j}\|^2.
\end{align}
Although this transformation is an additional approximation step regarding the original function $\psi_j$,
numerical experiments in, e.g.,~\cite{Kue16} do not indicate a substantial deterioration of the quality of the obtained shift
parameters and, moreover, it simplifies the evaluations of the functions and its derivatives a bit further.\\
Here we also handle the minimization of $\tilde\psi_j$ without the modification~\eqref{normmin_t}.
Methods based on the BFGS framework are capable of solving non-smooth optimization problems~\cite{LewO13,CurMo17} provided a careful implementation
is used that allows to deal with points where the objective function does not have derivatives. The \granso{} package~\cite{CurMo17} provides
\matlab{} implementations of these BFGS type methods and will also be tested in the numerical experiments for~\eqref{normmin} without the
modification~\eqref{normmin_t}.
\section{Multistep Extensions}\label{sec:multi}
Until now a single shift $\alpha_{j+1}$ was generated in each iteration step for reducing the residual norm from the current to the immediate next step.
Theoretically, this may be generalized towards the generation of shifts to be used in multiple, say $g>1$, future LR-ADI steps to reduce $\|W_{j+g}\|^2$
the most starting from $\|W_j\|^2$.
The NLS formulation for $g\geq1$
takes the form
\begin{align*}
 \lbrace\alpha_{j+1},\ldots,\alpha_{j+g}\rbrace&=\myargmin\limits_{\boldsymbol{\alpha}\in\C^{g}_-}\|\Psi_{j,j+g}(\boldsymbol{\alpha},\overline{
\boldsymbol { \alpha }} )\|^2,\quad \Psi_{j,j+g}(\boldsymbol{\alpha},\overline{\boldsymbol{
\alpha }} ):=\left(\prod\limits_{i=1}^{g}\cC(A,\boldsymbol{\alpha}(i))\right)W_j.
\end{align*}
Since we always assumed that if $\alpha_i\in\C_-$ also its complex conjugate is used (Remark~\ref{realadi}),
the above approach could yield parameters for up to $2g$ future LR-ADI steps. 
Obviously, solve this multistep optimization problem is harder than the single-step one. 
For instance, since the order in which the shifts are applied is not important we have $\Psi_{j,j+g}(\boldsymbol{\alpha},\overline{\boldsymbol{\alpha
}})=\Psi_{j,j+g}(\Pi_g\boldsymbol{\alpha},\Pi_g\overline{\boldsymbol{
\alpha }} )$ for any permutation $\Pi_r\in\R^{g\times g}$, implying that several local minima always exist. Moreover, the larger $g$, the harder it
will be to approximate
$\Psi_{j,j+g}$ by the data available at step $j$ such that potentially better shifts might be obtained from the single shift approach carried out $g$ times
in
succession. A similar generalization of~\eqref{normmin} can be found in~\cite{Kue16}, where no substantial improvements over the single shift approach are
reported.\\
\smallskip
A particular interesting special situation is when the $g>1$ future shift parameters are restricted to be equal, $\alpha_{j+i}=\alpha_{j+1}$, $1\leq i\leq g$.
We point out that similar multi-step approaches were investigated for Smith-type methods in, e.g.,~\cite{Pen00,morAntSG01,GugSA03,morSun08}.
Although this restriction will likely slow down the convergence compared to different shifts in each step, it can be practical for reducing the
computation time for solving the large-scale linear systems in the LR-ADI iteration. In particular, when sparse direct solvers are employed, a sparse LU
factorization
$LU=A+\alpha_{j+1}I$ 
is reused in the required forward and backward solves for the linear systems in the next $g$ iteration steps: $V_{j+i}=U^{-1}L^{-1}W_{j+i-1}$, $i\leq 1\leq g$.
This can be substantially cheaper compared to solving $g$ different shifted linear systems, depending on the value $g$ and the cost for solving a single
shifted
linear system. Hence, smaller overall computation times of LR-ADI can be achieved at the price of a slower Lyapunov residual reduction and larger generated low-rank factors. 
Obviously, one could simply use the shift obtained by the single step residual norm minimization framework $g$ times. 
We hope to obtain a better LR-ADI performance by incorporating the prior knowledge
that 
$\alpha_{j+1}$ is supposed to be used in $g\geq 1$ iteration steps in residual norm minimization approach. The associated multi-shift NLS formulation is
\begin{align*}
\alpha_{j+1}&=\myargmin\limits_{\alpha\in\C_-}\|\Psi_{j,j+g}(\alpha,\overline{\alpha} )\|^2,\quad \Psi_{j,j+g}(\alpha,\overline{\alpha}
):=\cC(A,\alpha)^g W_j.
\end{align*}
Using the product rule, the Jacobian and conjugate Jacobian of $\Psi_{j,j+r}$ are given by
\begin{align*}
 \frac{\partial \Psi_{j,j+g}(\alpha,\overline{\alpha})}{\alpha}&=-g\cC(A,\alpha)^g(A+\alpha
I)^{-1}W_{j},\\
 \frac{\partial
\Psi_{j,j+g}(\alpha,\overline{\alpha})}{\overline{\alpha}}&=-g\cC(A,\alpha)^{g-1}(A+\alpha
I)^{-1}W_{j}.
\end{align*}
By the same reasoning as in Section~\ref{sssec:approx_Krylov}, these formula indicate that for approximating the objective function and its derivatives, the
orders $p,m$
for the approximation subspace $\cE\cK_{p,m}(A,B)$ should be at least $g$, but in the numerical experiment smaller orders worked sufficiently well. 
The function minimization approach is extended in the same way by defining the scalar function $\psi_{j,j+g}(\nu,\xi):=\|\cC(A,\nu+\jmath \xi)^g W_j\|^2$.  
Theorem~\ref{thm:normfun_cale_grad} for the derivatives of $\psi_{j,j+g}$ can easily be
reformulated by using
\begin{align}
 W_{\alpha}^{(r)}:=g\cC(A,\alpha)^{g-1}W_j=g\left(I-2\nu L(\nu,\xi)^{-1}\right)^{g-1}W_j
\end{align}
instead of $W_{\alpha}^{(0)}$. 
\section{Numerical Experiments}\label{sec:numex}
In this section we execute several numerical examples to evaluate different aspects of the residual norm minimizing shift selection techniques.
All experiments were done in \matlab~2016a using a \intel\coretwo~i7-7500U CPU @ 2.7GHz with 16 GB RAM.
We wish to obtain an approximate solution such that the scaled Lyapunov residual norm satisfies
\begin{align*}
 \fR:=\|\cR^{\text{true}}\|/\|B\|^2\leq \varepsilon,\quad 0<\varepsilon\ll 1.
\end{align*}
Table~\ref{tab:ex} summarizes the used test examples. 
\begin{table}[t]
  \centering
  \caption{Overview of examples}\footnotesize
  \begin{tabularx}{\textwidth}{l|l|l|X|l}
    Example & $n$ & $s$&details&$\varepsilon$\\
    \hline
    \textit{cd2d}&40000&1, 5&finite difference discretization of 2d operator $\cL(u)=\Delta  u-100x\partderiv{u}{x}{}-1000y\partderiv{u}{y}{}$ on
$[0,1]^2$, homogeneous Dirichlet b.c.&1e-8\\
    \textit{cd3d}&27000&10&finite difference discretization of 3d operator $\cL(u)=\Delta
u-100x\partderiv{u}{x}{}-1000y\partderiv{u}{y}{}-10z\partderiv{u}{z}{}$ on
$[0,1]^3$, homogeneous Dirichlet b.c.&1e-8\\
    \textit{lung}&109460&10&model of 
  temperature and water vapor transport in the human lung from suitesparse collection~\cite{DavH11}&1e-8\\
    \textit{chip}&20082&5&finite element model of chip cooling process~\cite{morMooRGetal04}, $M=M^*\neq I$&1e-10\\
  \end{tabularx}\label{tab:ex}
\end{table}
The right hand side factors $B$ for all examples except \texttt{chip} are generated randomly with uniformly distributed entries, where the random number generator is
initialized by \texttt{randn{('state', 0)}} before each test. The maximal allowed number of LR-ADI steps is restricted to $150$.
In all experiments, we also emphasize the numerical costs for generating shift parameters by giving shift generation times
$t_{\text{shift}}$ next to the total run times $t_{\text{total}}$ of the LR-ADI iteration.
Before we compare the proposed residual minimizing shifts against other existing approaches, some tests with respect to certain aspects of this shift
selection framework are conducted.
\subsection{Approximation of the objective function}\label{ssec:approxopt}
At first, we evaluate different approximation approaches from Section~\ref{ssec:objapprox} for the objective functions, i.e., we test the influence of different
choices for the projection subspace to the overall performance of
the LR-ADI iteration. This experiment is carried out on the \texttt{cd2d} example with a single vector in $B$ and $\|B\|=1$. The NLS formulation~\eqref{NLS_C}
is employed and dealt with by the \tensorlab{} routine \texttt{nlsb\_gndl}.
As approximation subspaces the last $h=8$ columns of $Z_j$ from Section~\ref{sssec:approx_lrf} (denoted by $\cZ(8)$) and the extended Krylov approximations from
Section~\ref{sssec:approx_Krylov} with different orders $p,m$ are used. By the proposed subspace construction, the dimension of the basis is in all cases at
most 8. 
Moreover, the experiment is carried out in the single-step as well as multi-step fashion with $g=5$ from Section~\ref{sec:multi}. Figure~\ref{fig:objapprox}
and Table~\ref{tab:objapprox} summarize the obtained results.
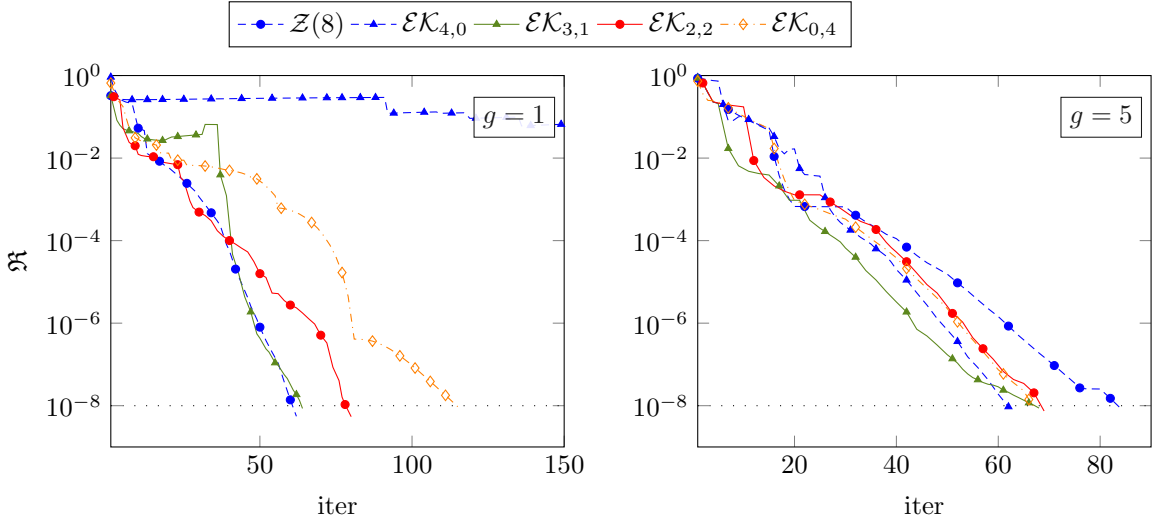
\begin{figure}[t]
\centering
\input{fdm2d_obj_approx1}~\input{fdm2d_obj_approx5}
  \caption{Residual norm history of LR-ADI iteration using different objective function approximations for \texttt{cd2d} example. Left: single step 
shift selection, right: multi step shift selection with $g=5$.}
  \label{fig:objapprox}
\end{figure}
\begin{table}[t]
  \centering
  \caption{Results with different projection subspaces for objective function approximation using the \texttt{cd2d} example. For single and multi step
approches ($g=5$), listed are the executed iteration numbers (iters), total and shift computation times ($t_{\text{total}}$, $t_{\text{shift}}$) in seconds, and
final residual norm $\fR$ (res).}
  \begin{tabularx}{\textwidth}{X|l|l|l|l||l|l|l|l}
    &\multicolumn{4}{c||}{single step ($g=1$)}&\multicolumn{4}{c}{multi step ($g=5$)}\\
  projection space&iters&$t_{\text{total}}$&$t_{\text{shift}}$&res&iters&$t_{\text{total}}$&$t_{\text{shift}}$&res\\
    \hline
$\cZ(8)$&62&19.0&3.1&5.5e-09&84&8.3&1.7&8.7e-09\\
$\cE\cK_{4,0}$&151&53.3&8.2&6.7e-02&62&6.9&1.3&9.4e-09\\
$\cE\cK_{3,1}$&64&22.4&4.8&8.6e-09&68&7.5&1.7&8.8e-09\\
$\cE\cK_{2,2}$&80&26.2&4.9&5.4e-09&69&7.5&1.7&7.4e-09\\
$\cE\cK_{0,4}$&115&38.3&7.8&9.5e-09&67&8.3&1.9&9.5e-09\\
  \end{tabularx}\label{tab:objapprox}
\end{table}
Apparently, for the single step optimization approach, using $\cZ(h)$ as approximation space seems to yield the best shifts
compared to the other subspace choices. The required iteration numbers and timings are the smallest among all tested settings. 
In particular, the pure Krylov ($m=0$) and inverse Krylov subspaces ($p=0$) lag behind the other choices.
The picture changes when considering the multistep optimization from Section~\ref{sec:multi} over $g=5$ steps, where the extended Krylov approximations of the
objective function yield better shifts, i.e., less iteration steps compared to using $\range{Z_j}$. Interestingly, in some cases the number of iteration steps
is
even lower compared to the single step optimization. Due to the reuse of LU factorizations over $g=5$ steps, and since the optimization problem has to be
solved less frequently  (only in every $5$th step or even less if the generated shift is complex), the savings in the computation times reported in
Table~\ref{tab:objapprox} are substantial. 
To conclude, while the standard objective function approximation using $\range{Z_j}$ seems to work satisfactory in most cases for the single step shift
selection, for $g>1$ better results might be obtained by the extended Krylov approximations proposed in this work.

\subsection{Choice of the optimization routine}\label{ssec:opt_solve}
\begin{table}[t]
  \centering
  \caption{Results with different optimization routines for the \texttt{cd3d} example.}
  \begin{tabularx}{\textwidth}{l|X|l|l|l|l}
  opt. problem&opt. routine&iters&$t_{\text{total}}$&$t_{\text{shift}}$&res\\
    \hline
NLS~\eqref{NLS_C}&\texttt{nlsb\_gndl}&52&98.3&2.4&2.9e-09\\
NLS~\eqref{normmin_t}&\texttt{nlsb\_gndl}&47&73.0&1.6&7.8e-09\\ 
fun.min.~\eqref{normmin}&\texttt{fmincon} with int. point method&48&70.8&2.5&5.6e-09\\
fun.min.~\eqref{normmin_t}&\texttt{fmincon} with int. point method&50&77.2&2.1&4.4e-09\\  
fun.min.~\eqref{normmin}&\texttt{fmincon} with thrust region reflective method&49&78.4&2.1&4.7e-09\\ 
fun.min.~\eqref{normmin_t}&\texttt{fmincon} with thrust region reflective method&51&76.9&1.7&2.0e-09\\  
fun.min.~\eqref{normmin}&\granso{}&49&89.2&14.5&7.3e-09\\ 
  \end{tabularx}\label{tab:optsolve}
\end{table}

Now we test different optimization problem formulations~\eqref{NLS_C},~\eqref{normmin} as well as different optimization routines for the
\texttt{cd3d} example having $s=10$ columns in $B$. As in the previous experiment, the \tensorlab{} routine \texttt{nlsb\_gndl} is used for the NLS problem,
but  for the function minimization problem~\eqref{normmin} we employ \granso{} and \texttt{fmincon}. In \texttt{fmincon} the interior-point and
trust-region-reflective methods are used as optimization routines. Since $s>1$, we also use the tangential approximation~\eqref{normmin_t} to avoid the
potential nonsmoothness of $\psi_j$ in~\eqref{normmin} and test this modification also within the NLS framework. The projection subspaces for the objective
function approximations are constructed from the last $h=4$ block columns of $Z_j$. Table~\ref{tab:optsolve} summarizes the results. Judging from the number of
required LR-ADI steps, the usage of different optimization routines appears to have less impact than working with different objective function approximations.
The additional tangential approximation~\eqref{normmin_t}
seems to slow down the LR-ADI iteration only marginally. The exception for this is when the NLS formulation~\eqref{NLS_C} and \texttt{nlsb\_gndl} is used, where
5 less LR-ADI steps, and consequently less computation time, are required. 
Using \granso{} resulted in comparatively high computational times for this shift generation.  The main computational bottleneck in this method are the arising
quadratic optimization problems.
Apparently, the non-smoothness of the function $\phi_j$ (in the sense of coalescing eigenvalues of $W(\alpha)^*W(\alpha)$) did hardly occur or appear to be
problematic for methods for smooth optimization problems, so that the application of non-smooth optimizers or using the tangential
approximation~\eqref{normmin_t} might not be necessary in most cases. 
Although not reported here, tests using \texttt{fmincon} without explicitly provided Hessians led to similar results regarding the
required number of steps of the LR-ADI iteration, but marginally longer shift generation times since the inherent optimization algorithms (interior-point or 
trust-region-reflective) required more steps. 
The performance of the optimization routines appeared to be also noticeably influenced by the choice of the initial guess. Using the heuristic instead of the residual-Hamiltonian selection for determining the initial guess led to higher shift generation times due to longer runs of the optimization routines.
Setting up the constraints~\eqref{GCALE_opti_bounds} for the optimization variables by using the computed Ritz values (eigenvalues of $H_j$) led in a few cases
to difficulties for the solution of the optimization problems. Especially the upper bounds for the imaginary parts of the shift parameters
appeared to be of strong influence. Further adjustments are necessary in this directions, also with respect to deciding in advance if the optimization problems
can be safely restricted to real variables. Currently, this is only done for problems with real spectra (e.g., $A=A^*$).
\subsection{Comparison with other shift selection routines and methods}\label{ssec:compar_shift}
Now the LR-ADI performance obtained with the approximate residual norm minimizing shifts is compared with the other shift selection strategies reviewed in
Section~\ref{ssec:shifts}. All employed shift generation and selection approaches are used and abbreviated as in Table~\ref{tab:shifts}.
\begin{table}[t]
  \centering
  \caption{Overview of employed shift selection strategies.}
  \begin{tabularx}{\textwidth}{c|l|X|l}
   type&abbreviation&description of strategy&info\\
   \hline
  \multirow{6}{*}{\rotatebox{90}
  {\mbox{precomputed}}}
  &heur($J,p,m$)&heuristic selection of $J\in\N$ shifts from Ritz values associated
to
$\cE\cK_{p,m}(A,B\mathbf{1}_s)$, cyclic usage&\cite{Pen99,Saa09}\\
&Wachs$(\epsilon,~p,~m)$ &Wachspress selection using Ritz values associated to $\cE\cK_{p,m}(A,B\mathbf{1}_s)$ and tolerance $0<\epsilon\ll1$, cyclic
usage&\cite{Saa09,Wac13,Sab07}\\
  \hline
\multirow{10}{*}{\rotatebox{90}
  {\mbox{adaptive}}}
  &$\cZ(h)$+heur&projection based shifts using newest $h$ block columns of $Z_j$ and selection via heuristic&\cite{BenKS14b,Kue16},
Section~\ref{sssec:projshifts}\\
&$\cZ(h)$+conv&projection based shifts as above, but  convex hull based selection&\cite{DruS11}, Section~\ref{sssec:convshifts}\\
&$\cZ(h)$+Hres&projection based shifts as above, but residual Hamiltonian based selection&\cite{BenBKetal18}, Section~\ref{sssec:Hres}\\
&resmin+$\cQ$+OS&residual norm minimizing shifts with $\cQ$ as approximation space and OS as optimization routine&\cite{BenKS14b,Kue16},
Section~\ref{sec:resminshift}\\
  \end{tabularx}
\end{table}\label{tab:shifts}

We also run a few tests with the multishift approach with $g=5$.
For each example, we
also compare the LR-ADI to the rational Krylov subspace method~\cite{DruKS11} (RKSM) equipped with the convex-hull based shift selection~\cite{DruS11}. The reduced Lyapunov equation
in RKSM is solved in every 5th step.

Table~\ref{tab:compare_shift} summarizes the results and Figure~\ref{fig:compare_shift} shows the history of the scaled Lyapunov residual norms for some selection approaches and the \texttt{cd3d}, \texttt{chip} examples.
The proposed residual norm minimizing shift generation strategy based on reduced objective functions leads to the smallest number of required iteration steps compared to the other selection approaches. The obtained rate of residual norm reduction is very close to the one obtained by RKSM, but LR-ADI required in all tests less computation time.
Hence, taking both the iteration numbers as well as the computation times in account, with the right set of shift parameters the LR-ADI iteration is competitive to RKSM. Note that RKSM is theoretically expected to converge faster than LR-ADI~\cite{DruKS11}.
Among the Ritz value based shift selection techniques (Section~\ref{sssec:projshifts}) for LR-ADI, the Residual-Hamiltonian selection~(Section~\ref{sssec:Hres},
$\cZ(h)$+Hres) appears to perform best, leading to iteration numbers close to the ones obtained with the residual minimizing shifts. The precomputed shift
approaches (heur($J,p,m$), Wachs($\epsilon,p,m$)) could in several cases not compete with the dynamic shift generation approaches, which, again underlines the
superiority of an adaptive selection of shift parameters. The generation times $t_{\text{shift}}$ of the adaptive shifts were in all cases only a small fraction
of the total computation times $t_{\text{total}}$. Due to the need to solve (compressed) optimization problems, the generation times of the residual minimizing
shifts was in some cases slightly higher compared to the other approaches.  Although the multishift selection approach yielded higher iteration numbers, they
led to a substantial reduction in the computation times $t_{\text{total}}$ because 
of the reuse of LU factorizations over several iteration steps.

\begin{table}[t]
  \centering
  \caption{Comparison of different shift routines and comparison against RKSM}
  \begin{tabularx}{\textwidth}{l|l|X|l|l|l|l}
  ex.&method&shift selection strategy&iters&$t_{\text{total}}$&$t_{\text{shift}}$&res\\
    \hline
\multirow{7}{*}{\rotatebox{90}
  {\mbox{\texttt{cd2d}}}}&
  \multirow{6}{*}{
  {\mbox{LR-ADI}}}
&heur(20, 30, 20)&137 & 35.3 & 0.8 &9.0e-09\\ 
&&wachs(10$^{-8}$, 30, 20)&93 & 24.0 & 0.8 &7.2e-09\\ 
&&$\cZ(4)$+heur&74 & 20.7 & 0.7 &9.2e-09\\ 
&&$\cZ(4)$+conv&80 & 21.5 & 2.2 &6.1e-09\\ 
&&$\cZ(4)$+Hres&74 & 19.4 & 1.6 &2.8e-09\\ 
&&resmin+$\cZ(4)$+\texttt{fmincon}&60 & 19.8 & 4.0 &3.0e-09\\ 
&RKSM&convex hull&61 & 39.0 & 4.5 &4.3e-09\\
\hline
\multirow{8}{*}{\rotatebox{90}
  {\mbox{\texttt{cd3d}}}}&
  \multirow{7}{*}{
  {\mbox{LR-ADI}}}
&heur(20, 40, 30)&68 & 85.8 & 1.9 &1.7e-09\\ 
&&wachs(1e-08, 30, 20)&150 & 188.2 & 1.4 &5.2e-04\\ 
&&$\cZ(4)$+heur&71 & 89.1 & 0.3 &1.9e-09\\ 
&&$\cZ(4)$+conv&57 & 70.5 & 1.3 &2.3e-09\\ 
&&$\cZ(4)$+Hres&52 & 64.9 & 1.1 &9.7e-09\\ 
&&resmin+$\cZ(4)$+\texttt{fmincon}&50 & 65.7 & 1.9 &6.7e-09\\ 
&&resmin+$\cE\cK_{1,1}$+\texttt{nlsb\_gndl}, $g=5$&59 & 23.4 & 0.9 &6.2e-09\\ 
&RKSM&convex hull&61 & 115.8 & 7.7 &5.6e-11\\ \hline
\multirow{8}{*}{\rotatebox{90}
  {\mbox{\texttt{lung}}}}&
  \multirow{7}{*}{
  {\mbox{LR-ADI}}}
&heur(20, 30, 20)&150 & 60.5 & 4.0 &1.7e-08\\ 
&&wachs(1e-8, 30, 20)&150 & 56.8 & 4.2 &2.8e-02\\ 
&&$\cZ(2)$+heur&94 & 37.5 & 0.9 &4.6e-10\\ 
&&$\cZ(2)$+conv&80 & 36.2 & 5.7 &3.6e-09\\ 
&&$\cZ(2)$+Hres&71 & 31.7 & 4.1 &8.5e-09\\ 
&&resmin+$\cZ(2)$+\texttt{fmincon}&65 & 29.5 & 3.5 &8.6e-09\\ 
&&resmin+$\cZ(2)$+\granso{}, $g=5$&69 & 8.4 & 1.7 &9.5e-09\\ 
&RKSM&convex hull&67 & 126.9 & 7.4 &1.9e-09\\ 
\hline
\multirow{8}{*}{\rotatebox{90}
  {\mbox{\texttt{chip}}}}&
  \multirow{7}{*}{
  {\mbox{LR-ADI}}}
&heur(10, 20, 10)&33 & 24.5 & 1.2 &7.1e-11\\ 
&&wachs(1e-12, 20, 10)&34 & 24.4 & 1.2 &5.0e-13\\ 
&&$\cZ(4)$+heur&70 & 49.8 & 0.3 &5.6e-11\\ 
&&$\cZ(4)$+conv&55 & 39.1 & 0.3 &5.6e-11\\ 
&&$\cZ(4)$+Hres&43 & 30.5 & 0.2 &8.6e-12\\ 
&&resmin+$\cZ(4)$+\texttt{fmincon}&32 & 27.0 & 0.9 &2.2e-11\\ 
&&resmin+$\cE\cK_{1,1}$+\texttt{nlsb\_gndl}, $g=5$&39 & 9.5 & 0.1 &7.8e-11\\ 
&RKSM&convex hull&32 & 25.7 & 1.7 &1.5e-12\\ 
  \end{tabularx}\label{tab:compare_shift}
\end{table}

\begin{figure}[t]
\centering
\input{fdm3d_compare_shift}
\input{chip_compare_shift}
  \caption{Residual norm history of LR-ADI iteration and RKSM using different shift selection strategies.}
  \label{fig:compare_shift}
\end{figure}
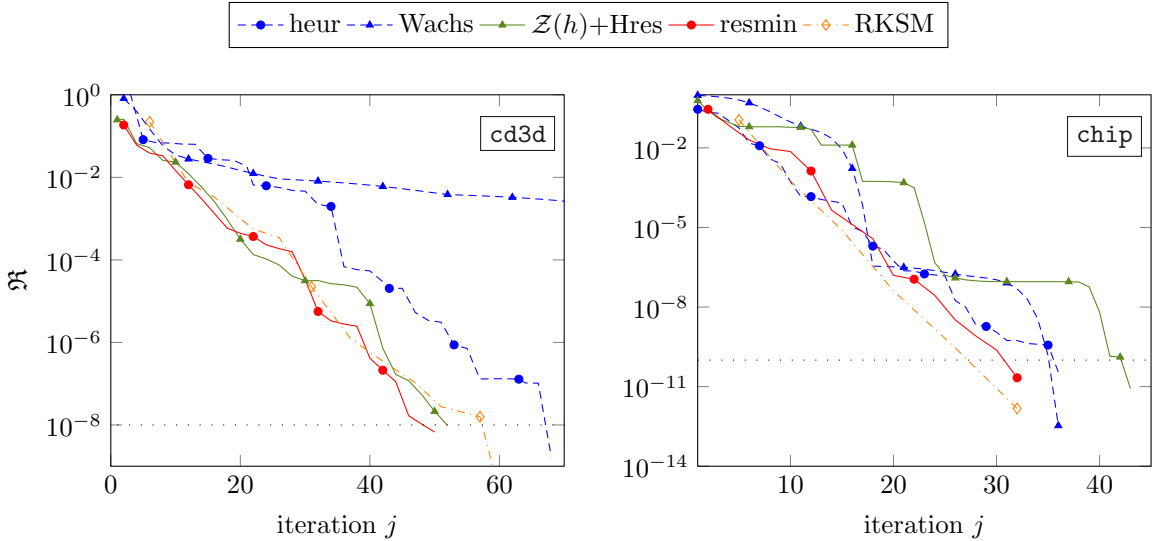
\section{Summary}\label{sec:summary}
This article discussed dynamically generated shift parameters for the LR-ADI iteration for large Lyapunov equations.
The selection of shifts was based on a residual norm minimization principle, which could be formulated as a nonlinear least squares or function minimization
problem. Since the involved objective functions are too expensive to evaluate, a framework using approximated objective functions was developed. These
approximations were built using projections onto low-dimensional subspaces, whose efficient construction from the data generated by the LR-ADI iteration was
presented. The numerical experiments showed that the proposed shift generation approach resulted in the fastest convergence of LR-ADI, bringing it very close to
the rational Krylov subspace method in terms of the iteration numbers. At the expense of higher iteration numbers, a substantial computation time reduction
could be achieved by a multishift selection approach.
Obvious future research direction might include similar shift selection strategies in LR-ADI type methods for other matrix equations, e.g., algebraic Sylvester
and Riccati equations, where first investigations can be found in~\cite{BenKS14b,Kue16,BenBKetal18} and should be further refined in future research. Deriving a similar multishift selection for RKSM is also an open topic.
Improving the solution of the occurring optimization problems by, e.g., providing better constraints or initial guesses, would further increase the performance
of the residual norm minimizing shift selection.

\section*{Acknowledgments}
  Thanks go to Tim Mitchell for helpful discussions regarding numerical optimization as well as for assistance with his
software \granso{}. Additional thanks go to Davide Palitta for proofreading.


\end{document}

%% file: fdm2d_obj_approx1.tex
%
%
%
%
\begin{tikzpicture}
\pgfplotsset{every axis title/.append style={at={(0.98,0.9)},anchor={north
        east},draw=black,fill=white,opacity=0.9}}
\begin{axis}[%
width=0.46\linewidth,
cycle list name=resN,
name=objapprox1,
title={$g=1$},
xmin=1,
xmax=150,
xlabel={iter},
ylabel={$\fR$},
ymode=log,
ymin=1e-9,
ymax=1,
yminorticks=false,
legend style={draw=black,fill=white,legend pos=north west,below right,
at={(0.26,1.2)},
legend columns=5},
legend entries={$\cZ(8)$,$\cE\cK_{4,0}$,$\cE\cK_{3,1}$,$\cE\cK_{2,2}$,$\cE\cK_{0,4}$},
]
\addplot 
  table[row sep=crcr]{1	0.325797645851315\\
2	0.35719740301062\\
4	0.26312563619185\\
6	0.227145895770247\\
8	0.212386586159755\\
10	0.0530783546586726\\
12	0.0466104915382267\\
13	0.012881410303664\\
14	0.0113591930962562\\
16	0.00883394815775651\\
17	0.0083938641630044\\
19	0.00638674944842941\\
20	0.00589379992259278\\
22	0.00439365724985575\\
24	0.00330188012514257\\
26	0.00244663216960515\\
27	0.00207035555282307\\
29	0.00150426430892609\\
30	0.00114530878859804\\
32	0.000740187826471065\\
34	0.000472472326250924\\
36	0.000299391694072705\\
37	0.000229129872448444\\
39	8.10759940467735e-05\\
40	5.63439725199519e-05\\
42	2.04131742902912e-05\\
44	9.06923448369857e-06\\
46	4.15360073994272e-06\\
47	2.57186689141577e-06\\
49	1.24037025813697e-06\\
50	7.99955486905606e-07\\
52	3.79024370509674e-07\\
54	1.98101527112938e-07\\
56	1.07384806700063e-07\\
58	3.76570092169729e-08\\
60	1.38720791643474e-08\\
62	5.53580257300501e-09\\
};

\addplot 
  table[row sep=crcr]{1	0.915927647885331\\
2	0.486575368156221\\
4	0.266341458877113\\
6	0.260536566797488\\
7	0.260078864682815\\
8	0.260071843911554\\
9	0.260187741433945\\
10	0.260307566356739\\
11	0.260419787157374\\
12	0.260588414686459\\
13	0.260819887437835\\
14	0.261097222550133\\
15	0.261402150051417\\
16	0.261735472363477\\
17	0.2622576506022\\
18	0.262695551739826\\
19	0.263159394293825\\
21	0.264270699864975\\
22	0.264740541491193\\
23	0.265216074694185\\
25	0.26633459398545\\
27	0.267452192539108\\
29	0.26862616595565\\
30	0.26921545043597\\
32	0.270279851143558\\
34	0.271232437343597\\
36	0.272276372877356\\
38	0.273603595650528\\
40	0.274801979050012\\
42	0.276011562598115\\
44	0.277205127838769\\
46	0.278289995132797\\
48	0.279333994247134\\
50	0.28028858530148\\
52	0.281331234386516\\
54	0.282320117331288\\
56	0.283186452844049\\
58	0.283986297125937\\
60	0.284750628512283\\
62	0.285657825010582\\
64	0.28644656947139\\
66	0.287216289322209\\
68	0.288149661593523\\
70	0.288892419299394\\
71	0.289257213163566\\
72	0.289592496555845\\
73	0.289900764513226\\
74	0.290185175321802\\
75	0.290430305470861\\
76	0.290664694309279\\
77	0.29092145843787\\
79	0.291358298000358\\
80	0.291635164018823\\
81	0.29190379562215\\
82	0.292185728679966\\
83	0.292503070673631\\
84	0.292834508670378\\
85	0.293173034641319\\
86	0.293516897381091\\
87	0.293811917875289\\
88	0.294123046727812\\
89	0.294434745003478\\
90	0.294789365495004\\
91	0.295123650768722\\
92	0.121548656930226\\
94	0.122815586179308\\
96	0.124309359164368\\
98	0.125765326743284\\
100	0.126961324540533\\
102	0.12806082040039\\
104	0.129248612313271\\
106	0.130344198325163\\
108	0.131329408852615\\
109	0.121402610380749\\
111	0.122479540309749\\
113	0.123349513518304\\
115	0.124300242424964\\
117	0.125236640032263\\
118	0.110568882205047\\
119	0.0917114842132782\\
121	0.0925496190083387\\
123	0.0934060123008527\\
125	0.0942427203084897\\
127	0.0950728506440382\\
129	0.095924665243979\\
131	0.096805586859222\\
133	0.0976133746074039\\
135	0.0983860603192847\\
136	0.0601148992601669\\
137	0.0605111950518719\\
139	0.0612849026084331\\
141	0.0621825302373328\\
143	0.0631168358069079\\
145	0.0639993858968391\\
147	0.0648863396537844\\
149	0.0657850445359844\\
151	0.0666715029236465\\
};

\addplot 
  table[row sep=crcr]{1	0.326557364870402\\
3	0.0824868115994812\\
4	0.0624283191065973\\
5	0.0517130644092859\\
6	0.0480262254459562\\
7	0.0457268452195517\\
8	0.0441283449750062\\
10	0.0395130103601436\\
11	0.033320812651134\\
12	0.0299782894132397\\
13	0.029041035803021\\
14	0.0282519990376977\\
15	0.0276168473421559\\
16	0.0271672518365742\\
17	0.0269051823563004\\
18	0.0267687029666322\\
19	0.027796956001014\\
20	0.0302181127500324\\
21	0.0324489895722493\\
22	0.0327508836925695\\
23	0.0332709815726624\\
24	0.0338002476160174\\
25	0.0343359438660601\\
26	0.0348603083929656\\
27	0.0353804462989247\\
29	0.0364157684138441\\
31	0.0374697458049235\\
32	0.0653878414673287\\
34	0.0651620475066039\\
36	0.0649310039612108\\
37	0.00392381586805325\\
39	0.0012309662017203\\
41	4.64022811733417e-05\\
43	1.61091661780258e-05\\
45	4.36104190182727e-06\\
47	1.8704632530514e-06\\
49	5.5017782190809e-07\\
50	4.31815767235093e-07\\
52	2.38008756797465e-07\\
53	1.98605041627517e-07\\
55	1.09343948025672e-07\\
56	9.51929550299331e-08\\
58	5.63034552578793e-08\\
59	4.56170332209093e-08\\
61	2.35518674728691e-08\\
62	1.85844795553737e-08\\
63	1.35274299804652e-08\\
64	8.58930435172956e-09\\
};

\addplot 
  table[row sep=crcr]{2	0.306941245391252\\
4	0.257453964165394\\
5	0.054226466951956\\
7	0.024220056762376\\
8	0.0211151098053372\\
9	0.0199088612344334\\
10	0.0123626302182092\\
11	0.0115445036983313\\
13	0.0109486157009318\\
14	0.0108414923130206\\
15	0.010783477921725\\
16	0.0120717667753482\\
17	0.00833702772092321\\
19	0.00771949783389059\\
21	0.00728618157943025\\
23	0.00691705053222728\\
24	0.00333308850564841\\
25	0.00279727955435767\\
27	0.000915118110702086\\
28	0.000631960629306893\\
30	0.00049088448586643\\
32	0.000405941083967083\\
34	0.000302926658465365\\
36	0.000173846018083815\\
38	0.000126148509412415\\
40	0.000100328545736496\\
42	8.22303574840944e-05\\
44	6.54648873087643e-05\\
46	5.3173494750871e-05\\
48	3.12320564938397e-05\\
50	1.58294279491491e-05\\
52	1.2802909029141e-05\\
54	5.35308305002106e-06\\
56	5.1878745484807e-06\\
58	3.52691497199732e-06\\
60	2.74730711906046e-06\\
62	2.35065402099566e-06\\
64	1.76677534595096e-06\\
66	1.28588445475481e-06\\
68	9.53285763926811e-07\\
70	5.11030969771047e-07\\
72	3.32305218739959e-07\\
74	8.54534123039738e-08\\
76	4.06114745227393e-08\\
77	1.62017942043148e-08\\
78	1.06752573679398e-08\\
80	5.41818164787189e-09\\
};

\addplot 
  table[row sep=crcr]{1	0.657540401119065\\
3	0.27159750280232\\
5	0.177560616346809\\
6	0.0963822961511231\\
7	0.0549830775278711\\
9	0.0309363889507113\\
11	0.0258659578373626\\
12	0.0244180851935545\\
13	0.0233267286598673\\
14	0.0223352453606388\\
16	0.0206651363338631\\
17	0.0134952511784072\\
18	0.0133067476797311\\
20	0.0131788925202597\\
21	0.00894541001531216\\
23	0.00889474257360319\\
25	0.00888170012669401\\
26	0.00672622395087415\\
28	0.00665807987781294\\
30	0.00657308893833495\\
32	0.00641639563406843\\
33	0.00630341626837192\\
35	0.00596854985388215\\
36	0.00576478604386016\\
38	0.00534083110309739\\
40	0.00502194300820828\\
42	0.00474603307044875\\
44	0.00435618967770437\\
45	0.00410300805212767\\
47	0.00346840794752314\\
49	0.00313685201956093\\
51	0.00244248647306365\\
53	0.00178386265564157\\
55	0.00121934118294726\\
56	0.000616652231766467\\
57	0.000610051438023733\\
59	0.000573957964134734\\
61	0.000512247952883891\\
63	0.000440336328468682\\
65	0.000358232746402819\\
67	0.000274578712684647\\
69	0.000199176706196018\\
71	0.000130320975031493\\
73	7.82846059000419e-05\\
75	4.02485227336721e-05\\
77	1.68238598275347e-05\\
79	4.46444963995009e-06\\
81	4.10998479906375e-07\\
83	4.05706316123593e-07\\
85	3.95574543883774e-07\\
87	3.67276585053853e-07\\
89	3.26111444755734e-07\\
91	2.78141033713501e-07\\
93	2.29329531069757e-07\\
95	1.82052080735533e-07\\
96	1.61511409923709e-07\\
97	1.42340625267132e-07\\
98	1.24988263716283e-07\\
99	1.09509648761577e-07\\
100	9.47772660460035e-08\\
101	8.1800875366629e-08\\
102	7.08645830740963e-08\\
103	6.0749677306038e-08\\
104	5.23746700060426e-08\\
105	4.49955303794135e-08\\
106	3.86837286323654e-08\\
107	3.31856235369143e-08\\
108	2.84957826003454e-08\\
109	2.42681105450084e-08\\
110	2.08067480621167e-08\\
111	1.78303379075822e-08\\
112	1.52995479368934e-08\\
113	1.30455667131113e-08\\
114	1.11153554932834e-08\\
115	9.49207020553361e-09\\
};
\addplot [color=black,loosely dotted,forget plot]
  table[row sep=crcr]{1	1e-08\\
150	1e-08\\
};
\end{axis}

%% file: fdm2d_obj_approx5.tex
%
%
%
%

\begin{axis}[%
width=0.46\linewidth,
title={$g=5$},
cycle list name=resN,
at={(objapprox1.north east)},
xshift=5em,
anchor=north west,
xmin=1,
xmax=90,
xlabel={iter},
ymode=log,
ymin=1e-9,
ymax=1,
yminorticks=false,
]
\addplot 
  table[row sep=crcr]{1	0.86176253300135\\
2	0.652345463392374\\
3	0.380813554905306\\
4	0.229375582433163\\
5	0.21175765728542\\
7	0.150044440374314\\
9	0.109085227412093\\
11	0.0817417180832647\\
13	0.0617068872307542\\
15	0.0468814107373084\\
16	0.010953535268617\\
17	0.00373787351514918\\
18	0.00172259093812899\\
19	0.00100044938593118\\
20	0.000672324901278712\\
22	0.000669498798910447\\
24	0.000666703751801564\\
26	0.000663939305406735\\
28	0.000661205012434195\\
30	0.000658500432726562\\
32	0.00041270761034191\\
34	0.000284326602330503\\
36	0.000204961169143537\\
38	0.000142770469555083\\
40	0.000111524736726745\\
42	6.95716708750131e-05\\
44	4.37023722838039e-05\\
46	2.91276988850394e-05\\
48	2.13123661846167e-05\\
50	1.53705825543659e-05\\
52	9.48017427723915e-06\\
54	6.01448819719283e-06\\
56	3.66635903260307e-06\\
58	2.25557146806726e-06\\
60	1.44396693536813e-06\\
62	8.49554996506465e-07\\
64	5.01327201056108e-07\\
66	3.08282313787355e-07\\
68	1.89073524592478e-07\\
70	1.1921338230103e-07\\
71	9.3955330872535e-08\\
72	7.38155257123864e-08\\
73	5.74871996680527e-08\\
74	4.52558188696564e-08\\
75	3.52777415020808e-08\\
76	2.69905710537005e-08\\
77	2.57825209882012e-08\\
78	2.55543986442514e-08\\
79	2.54562656683106e-08\\
80	2.53734019532093e-08\\
82	1.50811309352549e-08\\
84	8.66924081164034e-09\\
};

\addplot 
  table[row sep=crcr]{1	0.896358654125994\\
2	0.82929609805289\\
3	0.792589345191251\\
4	0.754373745676617\\
5	0.734376831818196\\
6	0.199307478440614\\
7	0.0797371164462575\\
8	0.081950831666362\\
9	0.0997889135282367\\
10	0.105292417102477\\
11	0.0859408053483851\\
12	0.07305086518394\\
13	0.0645914204922015\\
14	0.058095178330517\\
15	0.0535020632382451\\
16	0.0333408947298839\\
17	0.0168715852946052\\
18	0.012699080393461\\
19	0.0163872116359975\\
20	0.0166613551826202\\
21	0.00557712005685437\\
22	0.00398371247196828\\
23	0.0039144380938302\\
24	0.00374491576299909\\
25	0.00365358434044931\\
26	0.00109884063082811\\
27	0.000566024930346269\\
28	0.000374463944488945\\
29	0.000275929016110797\\
30	0.000220640052826084\\
31	0.000179712676702035\\
32	0.000148036942175315\\
33	0.000119392637537489\\
34	0.000100018199245126\\
35	8.28551417271557e-05\\
36	6.33096070680473e-05\\
37	5.22933090040073e-05\\
38	4.21035507771514e-05\\
39	3.0010318819079e-05\\
40	1.99256577869996e-05\\
42	1.10753683223752e-05\\
44	5.4340204492667e-06\\
46	2.74200851533807e-06\\
48	1.29038165096857e-06\\
50	6.84896007885315e-07\\
52	3.59227442415376e-07\\
54	1.52847536277949e-07\\
56	8.36303234236877e-08\\
58	4.23746274689785e-08\\
60	1.89734705407391e-08\\
62	9.42841456937383e-09\\
};

\addplot 
  table[row sep=crcr]{1	0.862859850461027\\
2	0.655974921847403\\
3	0.3853111366899\\
4	0.230426831360858\\
5	0.211279434985462\\
7	0.0170647223325451\\
9	0.00650759544072591\\
11	0.00478722263174331\\
13	0.00421730932966339\\
15	0.00388879739368598\\
17	0.0021158223092263\\
19	0.000952662406663747\\
21	0.000948936849907556\\
23	0.000311793313689714\\
25	0.000199604756446877\\
26	0.000166344414676627\\
27	0.000136483866183128\\
28	0.000113341723172697\\
29	9.27722719964624e-05\\
30	6.62214635563389e-05\\
32	3.9655703302319e-05\\
34	1.88200555515851e-05\\
36	1.11088225009406e-05\\
38	5.83882286067169e-06\\
40	3.2582200853199e-06\\
42	1.86230793365646e-06\\
44	7.12760539134803e-07\\
46	4.89039343576149e-07\\
48	2.94646408778753e-07\\
50	1.7148123062875e-07\\
51	1.37203515440569e-07\\
52	1.0757994387381e-07\\
53	7.67573842542945e-08\\
54	5.88535580161275e-08\\
55	4.79050491512111e-08\\
56	4.24743746510767e-08\\
57	3.6532842312499e-08\\
58	3.30842665537457e-08\\
59	3.15331777118375e-08\\
60	2.8659363678277e-08\\
61	2.37940780064818e-08\\
62	2.10948214667475e-08\\
63	1.80003047958288e-08\\
64	1.56981443831993e-08\\
65	1.35884912526595e-08\\
66	1.17229140378245e-08\\
67	1.02821184800625e-08\\
68	8.75652226507713e-09\\
};

\addplot 
  table[row sep=crcr]{2	0.658774167535689\\
4	0.231310386242174\\
6	0.201701561489085\\
8	0.186807515318631\\
10	0.174251233498597\\
12	0.00872187980566685\\
14	0.00337646858032485\\
16	0.00211035185285151\\
18	0.00157742511900287\\
20	0.00128947939204334\\
21	0.00128882184522737\\
22	0.00128816354286679\\
23	0.00128750449521528\\
24	0.0012868447125294\\
25	0.00128618420506831\\
27	0.000864317932319077\\
29	0.000606407256707294\\
31	0.00040822363567118\\
33	0.000291556670696853\\
35	0.00023633814766787\\
36	0.000186001833236904\\
37	0.000129905890144768\\
38	9.11043544313173e-05\\
39	6.779917188994e-05\\
40	5.27200015022225e-05\\
42	3.08549785137587e-05\\
44	1.63765641336076e-05\\
46	8.59778479814657e-06\\
48	5.36403825884287e-06\\
50	2.55268542325809e-06\\
51	1.72375357962876e-06\\
52	1.37335895939812e-06\\
53	1.06385009405904e-06\\
54	6.86514723714999e-07\\
55	4.43810634214505e-07\\
57	2.40596418705521e-07\\
59	1.34459137317656e-07\\
61	6.73432042764325e-08\\
63	4.31383503326261e-08\\
65	3.46195799901391e-08\\
67	2.04792917352866e-08\\
69	7.44595868171926e-09\\
};

\addplot 
  table[row sep=crcr]{1	0.731493656749155\\
2	0.361595094590853\\
3	0.250144221989765\\
4	0.244400368204636\\
5	0.232122681777026\\
7	0.165055827479075\\
9	0.119993761211739\\
11	0.089060564917239\\
13	0.0678246696993034\\
15	0.0514641684665076\\
16	0.0170678481757131\\
17	0.00669106466652593\\
18	0.00304027744851037\\
19	0.00159179533342932\\
20	0.000939904532292044\\
22	0.000756265736822034\\
24	0.000623071039936402\\
26	0.000501370083640964\\
28	0.0004013991041341\\
30	0.000332299710833644\\
32	0.000211838700364651\\
34	0.000138754853019733\\
36	9.10465311741939e-05\\
38	6.09308713134812e-05\\
40	3.82275936240455e-05\\
42	2.11195271407263e-05\\
44	1.18412143603543e-05\\
46	6.95848676383697e-06\\
48	4.04630996675769e-06\\
50	2.25419988703934e-06\\
52	1.0781954630515e-06\\
54	5.41251475804822e-07\\
56	2.84800385932055e-07\\
58	1.474138361516e-07\\
60	7.43046973189201e-08\\
61	5.89816056213114e-08\\
62	4.56744130995594e-08\\
63	3.59062955639879e-08\\
64	2.88635617275792e-08\\
65	2.29384616520653e-08\\
66	1.41633416057377e-08\\
67	9.47565377989525e-09\\
};
\addplot [color=black,loosely dotted,forget plot]
  table[row sep=crcr]{1	1e-08\\
150	1e-08\\
};
\end{axis}
\end{tikzpicture}%

%% file: fdm3d_compare_shift.tex
%
%
%
 \begin{tikzpicture}
\pgfplotsset{every axis title/.append style={at={(0.98,0.9)},anchor={north
        east},draw=black,fill=white,opacity=0.9}}
\begin{axis}[%
width=0.46\linewidth,
cycle list name=resN,
name=shiftcd3d,
title={\texttt{cd3d}},
ylabel={$\fR$},
xmin=0,
xmax=70,
xlabel={iteration $j$},
ymode=log,
ymin=1e-9,
ymax=1,
yminorticks=false,
legend style={draw=black,fill=white,legend pos=north west,below right,
at={(0.26,1.25)},
legend columns=5},
legend entries={heur,Wachs,$\cZ(h)$+Hres,resmin,RKSM},
]
\addplot 
  table[row sep=crcr]{1	1.10763349510512\\
3	1.08964020816818\\
5	0.081849432388713\\
7	0.0690174879043089\\
9	0.0676772813398115\\
11	0.0640607813615749\\
13	0.0632018536429315\\
15	0.0288469065269739\\
17	0.0266566601492399\\
19	0.0257471021199775\\
21	0.018514989355172\\
22	0.00648937342137748\\
24	0.00625028773742805\\
26	0.00567214428781085\\
28	0.00482946833127034\\
30	0.00461408159348399\\
32	0.00220559563092914\\
34	0.00197290932380614\\
36	6.82057849092707e-05\\
38	5.80708112471444e-05\\
40	5.40449746810844e-05\\
42	2.95468785681399e-05\\
43	2.03708907938484e-05\\
45	2.04786170998184e-05\\
47	5.32319644283765e-06\\
49	3.43006719140633e-06\\
51	3.07877734654503e-06\\
53	8.70554567029507e-07\\
55	7.16091215277371e-07\\
57	1.30195668853776e-07\\
59	1.31334635690468e-07\\
61	1.31945015334142e-07\\
63	1.28691285838718e-07\\
64	1.04125766521465e-07\\
66	1.02899496131205e-07\\
68	1.71717508938228e-09\\
};
\addplot 
  table[row sep=crcr]{2	0.813659996989789\\
4	0.394137331058712\\
6	0.147784807412999\\
8	0.0589413540516167\\
10	0.0349391515102778\\
12	0.0278265502306889\\
14	0.0240584758203573\\
16	0.0209272351117901\\
18	0.0176825441010208\\
20	0.0147077582355864\\
22	0.0123972836211694\\
24	0.010598305659158\\
26	0.00918636112615203\\
28	0.0088186787491181\\
30	0.00846031903655964\\
32	0.00809873248103173\\
34	0.007724221822473\\
36	0.0073295517805626\\
38	0.00691016288484464\\
40	0.00646537576636295\\
42	0.00600035994719084\\
44	0.00552677616736953\\
46	0.00505989982410392\\
48	0.00461435853158073\\
50	0.0042021123774768\\
52	0.00383156325003356\\
54	0.0037281296531802\\
56	0.00362410132916731\\
58	0.00351574244129706\\
60	0.00339977362597609\\
62	0.00327346362054883\\
64	0.00313486567114576\\
66	0.00298319385711674\\
68	0.00281926029223738\\
70	0.00264578548393086\\
72	0.00246735677438781\\
74	0.00228994431493285\\
76	0.00212004321928436\\
78	0.00196353516150492\\
80	0.00191947262637887\\
82	0.00187499634620281\\
84	0.00182851547980379\\
86	0.00177860207017077\\
88	0.00172402018495424\\
90	0.00166380470560784\\
92	0.00159739494737717\\
94	0.00152481674990743\\
96	0.00144687904395535\\
98	0.00136530416868352\\
100	0.001282666158076\\
102	0.00120203456383674\\
104	0.00112637664421509\\
106	0.00110484661459164\\
108	0.00108299012486696\\
110	0.00106001056848999\\
112	0.00103517173674632\\
114	0.00100781095905488\\
116	0.0009773773232846\\
118	0.000943499926762195\\
120	0.000906086074529902\\
122	0.000865435720724116\\
124	0.000822334597558566\\
126	0.000778063695784835\\
128	0.000734263662156214\\
130	0.000692646939996275\\
132	0.000680735704265582\\
134	0.000668610005033321\\
136	0.000655827180027801\\
138	0.000641974329398494\\
140	0.000626676573918448\\
142	0.000609619436284374\\
144	0.000590588115756643\\
146	0.000569523140680192\\
148	0.000546583950372789\\
150	0.000522199289119858\\
};
\addplot 
  table[row sep=crcr]{1	0.247145889398508\\
2	0.25708545965943\\
4	0.0656750928157702\\
6	0.052216525810698\\
8	0.0259004352254581\\
10	0.0233560829874364\\
12	0.0122375017765433\\
14	0.0058546512526331\\
16	0.00257159128531652\\
18	0.000937193797230441\\
20	0.000317191201060798\\
22	0.000134435036864406\\
24	0.000105276617197449\\
26	7.51395375800295e-05\\
28	4.02635269444984e-05\\
30	3.12721305788309e-05\\
32	3.15241653564656e-05\\
34	2.64129113872426e-05\\
36	2.49014448817974e-05\\
38	2.17030844336789e-05\\
40	8.71670401986563e-06\\
42	7.19668670409962e-07\\
44	1.67719980481483e-07\\
46	1.17690243733129e-07\\
48	5.44210107306916e-08\\
50	2.13534803872209e-08\\
52	9.74729371182566e-09\\
};
\addplot 
  table[row sep=crcr]{2	0.18522533414479\\
4	0.0601052953255656\\
6	0.0386286840625092\\
8	0.033448303643434\\
10	0.0146292200003597\\
12	0.00660686797343362\\
14	0.0030349780072845\\
16	0.00132946661790597\\
18	0.000583112102947888\\
20	0.00043995891268036\\
22	0.000368832903012376\\
24	0.000231926931555289\\
26	0.000187003761219075\\
28	0.00015742424485434\\
30	3.38740883501082e-05\\
32	5.62495670501021e-06\\
34	3.33428087261721e-06\\
36	2.79163192715441e-06\\
38	2.46612947418132e-06\\
40	4.15881442414929e-07\\
42	2.12166428748527e-07\\
44	1.11744415146888e-07\\
46	1.67025436549983e-08\\
48	1.04773759876008e-08\\
50	6.72093053380624e-09\\
};
\addplot 
  table[row sep=crcr]{6	0.222538705256773\\
12	0.00742645396051783\\
16	0.00330772015451442\\
22	0.000556107246462002\\
26	0.000348315718164078\\
31	2.21559215978823e-05\\
37	1.2464288033321e-06\\
41	4.49254325191192e-07\\
47	1.05009529584043e-07\\
51	2.77107528580881e-08\\
57	1.58159062482866e-08\\
61	5.60465687414973e-11\\
};

\addplot [color=black,loosely dotted,forget plot]
  table[row sep=crcr]{1	1e-08\\
150	1e-08\\
};
\end{axis}

%% file: chip_compare_shift.tex
%
%
%
%
\begin{axis}[%
width=0.46\linewidth,
cycle list name=resN,
name=shiftchip,
at={(shiftcd3d.north east)},
xshift=5em,
anchor=north west,
title={\texttt{chip}},
xmin=1,
xmax=45,
xlabel={iteration $j$},
ymode=log,
ymin=1e-14,
ymax=1,
yminorticks=false,
]
\addplot 
  table[row sep=crcr]{1	0.284612266863784\\
2	0.222382991479843\\
3	0.201243790547839\\
5	0.0565163369768556\\
6	0.0149493502475803\\
7	0.0120335002712559\\
8	0.00382640496823069\\
9	0.00257712698714985\\
10	0.000532843188021515\\
11	0.000147910920503546\\
12	0.000145009280067795\\
13	0.00011309186179389\\
15	8.36744057347603e-05\\
16	1.04108938119339e-05\\
17	7.12813380711603e-06\\
18	1.98049249178223e-06\\
19	1.52302633665737e-06\\
20	5.45443099439169e-07\\
21	2.31149968807318e-07\\
22	2.29239292879377e-07\\
23	1.76849025612521e-07\\
25	1.48595771267943e-07\\
26	1.75187228337567e-08\\
27	9.92241732743692e-09\\
28	2.13079034604319e-09\\
29	1.85446763623223e-09\\
30	1.12750291473479e-09\\
31	5.49663281529446e-10\\
32	5.46464909832991e-10\\
33	4.228667251569e-10\\
35	3.69724921108442e-10\\
36	3.55333310185348e-11\\
};
\addplot 
  table[row sep=crcr]{1	0.952578066844361\\
2	0.902036828503098\\
3	0.836574042619954\\
4	0.749196336220306\\
5	0.635537559875929\\
6	0.497144685200167\\
7	0.35211747629013\\
8	0.230204556960944\\
9	0.146345193141837\\
10	0.0946828280916068\\
11	0.0657736616641543\\
12	0.0506013065417372\\
13	0.0377129070377625\\
14	0.0229010766942828\\
15	0.00915403237814463\\
16	0.00167641064609381\\
17	7.03736029038205e-05\\
18	3.49364664452489e-07\\
19	3.38304437560078e-07\\
20	3.26717602612665e-07\\
21	3.12084168218323e-07\\
22	2.92806001708731e-07\\
23	2.67690975137119e-07\\
24	2.36690361540173e-07\\
25	2.03381808780123e-07\\
26	1.74364075236447e-07\\
27	1.53342837797346e-07\\
28	1.38218467311349e-07\\
29	1.24245188191837e-07\\
30	1.06596996398061e-07\\
31	8.10096853657161e-08\\
32	4.9049653909073e-08\\
33	1.95987514194982e-08\\
34	3.59231053080063e-09\\
35	1.49738376055859e-10\\
36	3.37526621656905e-13\\
};
\addplot 
  table[row sep=crcr]{1	0.607222037748559\\
2	0.252841990654916\\
3	0.129562891379669\\
4	0.092200053454374\\
5	0.0643918157473054\\
6	0.0625461090099315\\
7	0.062296684426368\\
8	0.0622234596142567\\
9	0.062202562762406\\
10	0.0594664710355709\\
11	0.0594616746078548\\
12	0.0500771297275784\\
13	0.0127169271708864\\
14	0.0127164015052577\\
15	0.0127162147734533\\
16	0.0127161440051881\\
17	0.000546681656860237\\
18	0.000546383435606058\\
19	0.000543628640935961\\
20	0.000523096508304541\\
21	0.000486457963740542\\
22	0.000316754598492995\\
23	1.55891630375289e-05\\
24	4.70978888200015e-07\\
25	1.59353684909663e-07\\
26	1.24635471210793e-07\\
27	1.05833530737354e-07\\
28	9.70834047417366e-08\\
29	9.23615435550479e-08\\
30	9.09009349314487e-08\\
31	9.06981826788146e-08\\
32	9.0309596653967e-08\\
33	9.02546143643041e-08\\
35	9.02467446113469e-08\\
36	9.02156769249708e-08\\
37	9.0092190275792e-08\\
38	8.76806468777381e-08\\
39	5.69713700664613e-08\\
40	6.44050690513838e-09\\
41	1.40730625515547e-10\\
42	1.29964149387121e-10\\
43	8.58952255174321e-12\\
};
\addplot 
  table[row sep=crcr]{2	0.284260578140807\\
4	0.0706635757508755\\
6	0.0197832408850999\\
8	0.00925069528117225\\
10	0.00727109371810785\\
12	0.00134159357484208\\
14	4.46142642895117e-05\\
16	1.28603430677636e-05\\
18	3.74148757582666e-06\\
20	1.55578195739859e-07\\
22	1.11878193042905e-07\\
24	2.77616421381693e-08\\
26	3.26563982995167e-09\\
28	8.08238652242424e-10\\
30	2.40258509254415e-10\\
32	2.15841236972625e-11\\
};
\addplot 
  table[row sep=crcr]{5	0.113652964497825\\
10	0.00056821737190042\\
15	7.93809609131433e-06\\
20	4.11636947033642e-08\\
26	2.85508781286251e-10\\
32	1.50460359469529e-12\\
};

\addplot [color=black,loosely dotted,forget plot]
  table[row sep=crcr]{1	1e-10\\
150	1e-10\\
};

\end{axis}
\end{tikzpicture}%

%% file: Kue18_ADIresmin-arxiv.bbl
\begin{thebibliography}{10}

\bibitem{morAntSG01}
A.~C. Antoulas, D.~C. Sorensen, and S.~Gugercin.
\newblock A survey of model reduction methods for large-scale systems.
\newblock {\em Contemp. Math.}, 280:193--219, 2001.

\bibitem{AntSZ02}
A.~C. Antoulas, D.~C. Sorensen, and Y.~Zhou.
\newblock On the decay rate of {H}ankel singular values and related issues.
\newblock {\em Syst. Cont. Lett.}, 46(5):323--342, 2002.

\bibitem{BakES14}
J.~Baker, M.~Embree, and J.~Sabino.
\newblock {Fast singular value decay for {L}yapunov solutions with nonnormal
  coefficients}.
\newblock {\em {{SIAM} J. Matrix Anal. Appl.}}, 36(2):656--668, 2015.

\bibitem{BecT17}
B.~Beckermann and A.~Townsend.
\newblock {On the Singular Values of Matrices with Displacement Structure}.
\newblock {\em {{SIAM} J. Matrix Anal. Appl.}}, 38(4):1227--1248, 2017.

\bibitem{BenB16}
P.~Benner and Z.~Bujanovi\'c.
\newblock On the solution of large-scale algebraic {R}iccati equations by using
  low-dimensional invariant subspaces.
\newblock {\em Linear Algebra Appl.}, 488:430--459, 2016.

\bibitem{BenBKetal18}
P.~Benner, Z.~Bujanovi{\'c}, P.~K{\"u}rschner, and J.~Saak.
\newblock {RADI}: A low-rank {ADI}-type algorithm for large scale algebraic
  {R}iccati equations.
\newblock {\em Numer. Math.}, 138(2):301--330, Feb. 2018.

\bibitem{BenKS13}
P.~Benner, P.~K{\"u}rschner, and J.~Saak.
\newblock Efficient handling of complex shift parameters in the low-rank
  {C}holesky factor {ADI} method.
\newblock {\em Numer. Algorithms}, 62(2):225--251, 2013.

\bibitem{morBenKS13}
P.~Benner, P.~K{\"u}rschner, and J.~Saak.
\newblock An improved numerical method for balanced truncation for symmetric
  second order systems.
\newblock {\em Math. Comput. Model. Dyn. Syst.}, 19(6):593--615, 2013.

\bibitem{BenKS14b}
P.~Benner, P.~K{\"u}rschner, and J.~Saak.
\newblock Self-generating and efficient shift parameters in {ADI} methods for
  large {L}yapunov and {S}ylvester equations.
\newblock {\em Electron. Trans. Numer. Anal.}, 43:142--162, 2014.

\bibitem{BenS16}
P.~Benner and J.~Saak.
\newblock {Numerical solution of large and sparse continuous time algebraic
  matrix {R}iccati and {L}yapunov equations: a state of the art survey}.
\newblock {\em GAMM Mitteilungen}, 36(1):32--52, August 2013.

\bibitem{morCasPL17}
A.~Castagnotto, H.~K.~F. Panzer, and L.~B.
\newblock Fast $\mathcal{H}_2$-optimal model order reduction exploiting the
  local nature of krylov-subspace methods.
\newblock In {\em European Control Conference 2016}, pages 1958--1963, Aalborg,
  Denmark, 2016.

\bibitem{ColL96}
T.~Coleman and Y.~Li.
\newblock {An Interior Trust Region Approach for Nonlinear Minimization Subject
  to Bounds}.
\newblock {\em {SIAM} J. Optim.}, 6(2):418--445, 1996.

\bibitem{CurMo17}
F.~E. Curtis, T.~Mitchell, and M.~L. Overton.
\newblock A {BFGS-SQP} method for nonsmooth, nonconvex, constrained
  optimization and its evaluation using relative minimization profiles.
\newblock {\em Optim. Methods Softw.}, 32(1):148--181, 2017.

\bibitem{DavH11}
T.~A. Davis and Y.~Hu.
\newblock {The University of Florida Sparse Matrix Collection}.
\newblock {\em ACM Trans. Math. Softw.}, 38(1):1:1--1:25, Dec. 2011.

\bibitem{DruKS11}
V.~Druskin, L.~Knizhnerman, and V.~Simoncini.
\newblock Analysis of the rational {K}rylov subspace and {ADI} methods for
  solving the {L}yapunov equation.
\newblock {\em {SIAM} J. Numer. Anal.}, 49(5):1875--1898, 2011.

\bibitem{DruS11}
V.~Druskin and V.~Simoncini.
\newblock Adaptive rational {K}rylov subspaces for large-scale dynamical
  systems.
\newblock {\em Syst. Cont. Lett.}, 60(8):546--560, 2011.

\bibitem{FlaG12}
G.~M. Flagg and S.~Gugercin.
\newblock {On the {ADI} method for the {S}ylvester equation and the optimal
  $\mathcal{H}_2$ points}.
\newblock {\em Appl. Numer. Math.}, 64(0):50--58, 2013.

\bibitem{FroLS17}
A.~Frommer, K.~Lund, and D.~B. Szyld.
\newblock {Block Krylov subspace methods for computing functions of matrices
  applied to multiple vectors}.
\newblock {\em {Electr. Trans. Num. Anal.}}, 47:100--126, 2017.

\bibitem{FroS08}
A.~Frommer and V.~Simoncini.
\newblock {Matrix Functions}.
\newblock In W.~Schilders, H.~A. van~der Vorst, and J.~Rommes, editors, {\em
  {Model Order Reduction: Theory, Research Aspects and Applications}},
  volume~13 of {\em {Mathematics in Industry}}, pages 275--303. Springer Berlin
  Heidelberg, 2008.

\bibitem{Gra04}
L.~Grasedyck.
\newblock Existence of a low rank or {$H$}-matrix approximant to the solution
  of a {S}ylvester equation.
\newblock {\em Numer. Lin. Alg. Appl.}, 11(4):371--389, 2004.

\bibitem{GugSA03}
S.~Gugercin, D.~C. Sorensen, and A.~C. Antoulas.
\newblock A modified low-rank {S}mith method for large-scale {L}yapunov
  equations.
\newblock {\em Numer. Algorithms}, 32(1):27--55, 2003.

\bibitem{Gue13}
S.~G{\"u}ttel.
\newblock Rational {K}rylov approximation of matrix functions: {N}umerical
  methods and optimal pole selection.
\newblock {\em GAMM-Mitteilungen}, 36(1):8--31, 2013.

\bibitem{Kni92}
L.~A. Knizhnerman.
\newblock Calculation of functions of unsymmetric matrices using {A}rnoldi's
  method.
\newblock {\em Computational Mathematics and Mathematical Physics}, 31(1):1--9,
  1992.

\bibitem{Kue16}
P.~K{\"u}rschner.
\newblock {\em Efficient Low-Rank Solution of Large-Scale Matrix Equations}.
\newblock {D}issertation, Otto-von-Guericke-Universit{\"a}t, Magdeburg,
  Germany, Apr. 2016.
\newblock Shaker Verlag, ISBN 978-3-8440-4385-3.

\bibitem{Lan64}
P.~Lancaster.
\newblock On eigenvalues of matrices dependent on a parameter.
\newblock {\em Numer. Math.}, 6:377--387, 1964.

\bibitem{Lew97}
A.~Lewis.
\newblock {The mathematics of eigenvalue optimization}.
\newblock {\em Mathem. Prog.}, 97(1-2):155--176, 2003.

\bibitem{LewO13}
A.~S. Lewis and M.~L. Overton.
\newblock Nonsmooth optimization via quasi-{N}ewton methods.
\newblock {\em Math. Program.}, 141(1--2, Ser. A):135--163, 2013.

\bibitem{morLi00}
J.-R. Li.
\newblock {\em Model Reduction of Large Linear Systems via Low Rank System
  {G}ramians}.
\newblock {Ph.D. Thesis}, Massachusettes Institute of Technology, Sept. 2000.

\bibitem{LiW02}
J.-R. Li and J.~White.
\newblock Low rank solution of {L}yapunov equations.
\newblock {\em {SIAM} J. Matrix Anal. Appl.}, 24(1):260--280, 2002.

\bibitem{Men13}
E.~Mengi.
\newblock {A Support Function Based Algorithm for Optimization with Eigenvalue
  Constraints}.
\newblock {\em {SIAM} J. Optim.}, 27(1):246--268, 2017.

\bibitem{MenYK14}
E.~Mengi, E.~A. Yildirim, and M.~Kili\c{c}.
\newblock {Numerical Optimization of Eigenvalues of {H}ermitian Matrix
  Functions}.
\newblock {\em {{SIAM} J. Matrix Anal. Appl.}}, 35(2):699--724, 2014.

\bibitem{morMooRGetal04}
C.~Moosmann, E.~B. Rudnyi, A.~Greiner, and J.~G. Korvink.
\newblock Model order reduction for linear convective thermal flow.
\newblock In {\em THERMINIC 2004}, pages 317--321, Sophia Antipolis, France,
  2004.

\bibitem{NocW99}
J.~Nocedal and S.~J. Wright.
\newblock {\em {N}umerical {O}ptimization.}
\newblock Springer, New York, 1999.

\bibitem{Ove92}
M.~L. Overton.
\newblock {Large-Scale Optimization of Eigenvalues}.
\newblock {\em SIAM J. Optimiz.}, 2(1):88--120, 1992.

\bibitem{Pen99}
T.~Penzl.
\newblock {A cyclic low rank {S}mith method for large sparse {L}yapunov
  equations}.
\newblock {\em {{SIAM} J. Sci. Comput.}}, 21(4):1401--1418, 2000.

\bibitem{Pen00}
T.~Penzl.
\newblock Eigenvalue decay bounds for solutions of {L}yapunov equations: the
  symmetric case.
\newblock {\em Syst. Cont. Lett.}, 40:139--144, 2000.

\bibitem{Rem91}
R.~Remmert.
\newblock {\em {Theory of Complex Functions}}.
\newblock Springer-Verlag New York, 1991.

\bibitem{Ruh94c}
A.~Ruhe.
\newblock {The {R}ational {K}rylov algorithm for nonsymmetric Eigenvalue
  problems. {III}: Complex shifts for real matrices}.
\newblock {\em {BIT}}, 34:165--176, 1994.

\bibitem{Saa92}
Y.~Saad.
\newblock {\em Numerical Methods for Large Eigenvalue Problems}.
\newblock Manchester University Press, Manchester, UK, 1992.

\bibitem{Saa09}
J.~Saak.
\newblock {\em Efficient Numerical Solution of Large Scale Algebraic Matrix
  Equations in {PDE} Control and Model Order Reduction}.
\newblock {D}issertation, Technische Universit{\"a}t Chemnitz, Chemnitz,
  Germany, July 2009.

\bibitem{SaaKB16-mmess-1.0.1}
J.~Saak, M.~K\"{o}hler, and P.~Benner.
\newblock {M-M.E.S.S.-1.0.1} -- the matrix equations sparse solvers library,
  Apr. 2016.
\newblock see also:~\url{https://www.mpi-magdeburg.mpg.de/projects/mess}.

\bibitem{Sab07}
J.~Sabino.
\newblock {\em Solution of Large-Scale {L}yapunov Equations via the Block
  Modified {S}mith Method}.
\newblock PhD thesis, Rice University, Houston, Texas, June 2007.

\bibitem{Sim07}
V.~Simoncini.
\newblock A new iterative method for solving large-scale {Lyapunov} matrix
  equations.
\newblock {\em {SIAM} J. Sci. Comput.}, 29(3):1268--1288, 2007.

\bibitem{Sim16a}
V.~Simoncini.
\newblock Computational methods for linear matrix equations.
\newblock {\em {SIAM} Rev.}, 38(3):377--441, 2016.

\bibitem{SorVL12}
L.~Sorber, M.~{Van Barel}, and L.~{De Lathauwer}.
\newblock {Unconstrained Optimization of Real Functions in Complex Variables}.
\newblock {\em {SIAM} J. Optim.}, 22(3):879--898, 2012.

\bibitem{morSun08}
K.~Sun.
\newblock {\em Model order reduction and domain decomposition for large-scale
  dynamical systems}.
\newblock ProQuest LLC, Ann Arbor, MI, 2008.
\newblock Thesis (Ph.D.)--Rice University.

\bibitem{TruV07}
N.~Truhar and K.~Veseli{\'c}.
\newblock Bounds on the trace of a solution to the {L}yapunov equation with a
  general stable matrix.
\newblock {\em Syst. Cont. Lett.}, 56(7--8):493--503, 2007.

\bibitem{tensorlab30}
N.~Vervliet, O.~Debals, L.~Sorber, M.~{Van Barel}, and L.~{De Lathauwer}.
\newblock {Tensorlab 3.0}, 2016.
\newblock Available online at \url{http://www.tensorlab.net/}.

\bibitem{Wac13}
E.~L. Wachspress.
\newblock {\em The {ADI} Model Problem}.
\newblock Springer New York, 2013.

\bibitem{WalMNetal06}
R.~Waltz, J.~Morales, J.~Nocedal, and D.~Orban.
\newblock {An interior algorithm for nonlinear optimization that combines line
  search and trust region steps}.
\newblock {\em Math. Program.}, 107(3):391--408, Jul 2006.

\bibitem{morWol15}
T.~Wolf.
\newblock {\em ${H}_2$ Pseudo-Optimal Model Order Reduction}.
\newblock {D}issertation, Technische Universit\"at M\"unchen, Munich, Germany,
  2015.

\bibitem{PanW13a}
T.~Wolf and H.~K.-F. Panzer.
\newblock {The {ADI} iteration for {L}yapunov equations implicitly performs
  {H}2 pseudo-optimal model order reduction}.
\newblock {\em Internat. J. Control}, 89(3):481--493, 2016.

\end{thebibliography}
